\documentclass[oneside, a4paper, 11pt]{amsart}

\usepackage[utf8]{inputenc}
\usepackage[T1]{fontenc}

\usepackage{amsfonts, amsmath,amsthm,amssymb,bbm,comment,mathrsfs}

\usepackage[ttscale=.85]{libertine}
\usepackage{libertinust1math}
\DeclareSymbolFont{cmarrows}{OMS}{cmsy}{m}{n}
\DeclareMathSymbol{\mapstochar}{\mathrel}{cmarrows}{"37}

\usepackage[tracking=true]{microtype}
\SetTracking
  {encoding=T1, shape=sc}
  {10}
\SetProtrusion
  {encoding=T1,size={7,8}}
  {1={ ,750},2={ ,500},3={ ,500},4={ ,500},5={ ,500},
    6={ ,500},7={ ,600},8={ ,500},9={ ,500},0={ ,500}}

\newcommand*{\cat}[1]{\ensuremath{\mathcal{#1}}} 
\newcommand*{\functor}[1]{#1}
\newcommand*{\projectives}[1]{\proj(#1)}

\renewcommand{\colon}{\!\nobreak\mskip2mu\mathpunct{}\nonscript%
\mkern-\thinmuskip{:}\mskip6muplus1mu\relax%
}
\newcommand{\from}{\colon}
\newcommand*{\defeq}{\mathrel{\vcenter{\baselineskip0.5ex \lineskiplimit0pt
    \hbox{\scriptsize.}\hbox{\scriptsize.}}}%
=}

\newcommand*{\modc}[2][\cat{C}]{\ensuremath{\mathrm{mod}_{#1}(#2)}}

\newcommand*{\Alg}[1]{\ensuremath{\mathrm{Alg}(#1)}}
\newcommand*{\blank}{{-}}

\newcommand*{\leftdual}[1]{\leftidx{^\vee}{\!#1}{}}     
\newcommand*{\rightdual}[1]{{#1}^\vee}                   
\newcommand*{\ev}{\ensuremath{\mathrm{ev}}}
\newcommand*{\coev}{\ensuremath{\mathrm{coev}}}

\renewcommand*{\to}{\longrightarrow}
\renewcommand*{\mapsto}{\longmapsto}

\newcommand{\Fib}{\mathsf{Fib}} 

\renewcommand*{\hom}[3][\cat{M}]{\ensuremath{\underline{\mathrm{Hom}}_{#1}(#2,#3)}}
\newcommand*{\iend}[2][\cat{M}]{\ensuremath{\underline{\mathrm{End}}_{#1}(#2)}}
\newcommand*{\ienda}[2][A]{\ensuremath{\underline{\mathrm{End}}_{#1}(#2)}}
\usepackage{scalerel}   
\makeatletter
\let\@tl\triangleleft
\let\@tr\triangleright
\newcommand*{\@smallTriangle}[2]{\vcenter{\hbox{\scalebox{0.75}{\ensuremath{#1#2}}}}}
\newcommand*{\@medTriangle}[2]{\vcenter{\hbox{\scalebox{0.90}{\ensuremath{#1#2}}}}}
\newcommand*{\lact}{\mathbin{\mathpalette\@smallTriangle\@tr}}
\newcommand*{\ract}{\mathbin{\mathpalette\@smallTriangle\@tl}}
\let\@btl\blacktriangleleft
\let\@btr\blacktriangleright
\newcommand*{\blact}{\mathbin{\mathpalette\@medTriangle\@btr}}
\newcommand*{\bract}{\mathbin{\mathpalette\@medTriangle\@btl}}
\renewcommand*{\triangleright}{\lact}
\renewcommand*{\triangleleft}{\ract}
\renewcommand*{\blacktriangleright}{\blact}
\renewcommand*{\blacktriangleleft}{\bract}

\DeclareRobustCommand{\SkipTocEntry}[5]{}

\numberwithin{equation}{section}

\numberwithin{figure}{section}

\usepackage{thmtools,mathtools}
\usepackage{graphicx,enumitem,stmaryrd}
\usepackage{multicol}
\usepackage{xparse}
\usepackage{leftidx}    

\usepackage{anyfontsize}
\allowdisplaybreaks[4]
\usepackage[table]{xcolor}
\usepackage[normalem]{ulem}
\usepackage{bm, tensor}
\usepackage{tikz,tikz-cd}
\usetikzlibrary{matrix}
\usetikzlibrary{fit}
\usetikzlibrary{calc}
\tikzcdset{
  arrow style=tikz,
  diagrams={>={Straight Barb}}
}

\definecolor{blue}{rgb}{0.38, 0.51, 0.71}
\definecolor{red}{RGB}{175, 49, 39}
\definecolor{green}{RGB}{146, 227, 95}
\usepackage[english]{babel}
\usepackage[colorlinks, linktoc=page, linkcolor={red}, citecolor={green!65!black}, urlcolor={blue}]{hyperref}
\addto\extrasenglish{%
}

\makeatletter
\def\slashedarrowfill@#1#2#3#4#5{%
$\m@th\thickmuskip0mu\medmuskip\thickmuskip\thinmuskip\thickmuskip
  \relax#5#1\mkern-7mu%
  \cleaders\hbox{$#5\mkern-2mu#2\mkern-2mu$}\hfill
  \mathclap{#3}\mathclap{#2}%
  \cleaders\hbox{$#5\mkern-2mu#2\mkern-2mu$}\hfill
  \mkern-7mu#4$%
}
\def\rightslashedarrowfill@{%
\slashedarrowfill@\relbar\relbar\mapstochar\rightarrow}
\newcommand\xslashedrightarrow[2][]{%
\ext@arrow 0055{\rightslashedarrowfill@}{#1}{#2}}
\makeatother

\makeatletter
\newcommand{\ostar}{\mathbin{\mathpalette\make@circled\star}}
\newcommand{\make@circled}[2]{%
\ooalign{$\m@th#1\smallbigcirc{#1}$\cr\hidewidth$\m@th#1#2$\hidewidth\cr}%
}
\newcommand{\smallbigcirc}[1]{%
\vcenter{\hbox{\scalebox{0.77778}{$\m@th#1\bigcirc$}}}%
}
\makeatother

\usepackage{adjustbox}

\usepackage{CJKutf8}

\usetikzlibrary{backgrounds}
\usetikzlibrary{arrows}
\usetikzlibrary{shapes,shapes.geometric,shapes.misc}
\tikzstyle{tikzfig}=[baseline=-0.25em,scale=0.5]
\pgfkeys{/tikz/tikzit fill/.initial=0}
\pgfkeys{/tikz/tikzit draw/.initial=0}
\pgfkeys{/tikz/tikzit shape/.initial=0}
\pgfkeys{/tikz/tikzit category/.initial=0}
\pgfdeclarelayer{edgelayer}
\pgfdeclarelayer{nodelayer}
\pgfsetlayers{background,edgelayer,nodelayer,main}
\tikzstyle{none}=[inner sep=0mm]
\tikzstyle{every loop}=[]
\tikzstyle{whitedot}=[fill=white, draw, shape=circle, scale=0.3, tikzit draw=black, tikzit shape=circle, tikzit fill=white]
\tikzstyle{blackdot}=[fill=black, draw, shape=circle, scale=0.3, tikzit draw=black, tikzit shape=circle, tikzit fill=black]
\tikzstyle{box}=[fill=white, draw=black, shape=rectangle, tikzit fill=white]
\tikzstyle{BL}=[draw=black, shape=circle, fill=black, scale=0.3]
\tikzstyle{PP}=[draw={rgb,255:red,102;green,41;blue,163}, shape=circle, fill={rgb,255:red,102;green,41;blue,163}, scale=0.3]
\tikzstyle{morphism-edge}=[-, draw=black, thick]
\tikzstyle{cotensor}=[-, draw=gray]
\tikzstyle{braid-over}=[-, draw=white, thick, double=black, double distance=0.8pt, tikzit draw={rgb,255: red,128; green,0; blue,128}]
\tikzstyle{purple-over}=[-, draw=white, thick, double={rgb,255:red,102;green,41;blue,163}, double distance=0.8pt, tikzit draw={rgb,255:red,102;green,41;blue,163}]
\tikzstyle{purple}=[-, draw={rgb,255:red,102;green,41;blue,163}, thick]
\tikzstyle{blue-under}=[-, draw={rgb,255:red,0;green,128;blue,128}, thick]
\tikzstyle{ddd}=[-, draw=black, dash dot dot]
\tikzstyle{unit}=[-, draw=black, densely dotted]
\tikzstyle{Front}=[-, draw=black, fill={rgb,255; red,255; green,255; blue,255}, opacity=0.8]
\tikzstyle{Hidden}=[-, draw=black, fill={rgb,255; red,255; green,255; blue,255}, opacity=0.2]
\tikzstyle{directed}=[-, thick, black, decoration={markings, mark=at position 0.5 with {\arrow{>}}}, postaction=decorate]
\newcommand{\tikzfig}[1]{{%
    \tikzstyle{every picture}=[tikzfig]
    \IfFileExists{#1.tikz}
    {\input{#1.tikz}}
    {%
      \IfFileExists{figures/#1.tikz}
      {\input{figures/#1.tikz}}
      {\tikz[baseline=-0.5em]{\node[draw=red,font=\color{red},fill=red!10!white] {\textit{#1}};}}%
    }%
}}

\tikzcdset{scale cd/.style={every label/.append style={scale=#1},
  cells={nodes={scale=#1}}}}

\calclayout

\usetikzlibrary{calc}
\usetikzlibrary{decorations.pathmorphing}
\tikzset{curve/.style={settings={#1},to path={(\tikztostart)
  .. controls ($(\tikztostart)!\pv{pos}!(\tikztotarget)!\pv{height}!270:(\tikztotarget)$)
  and ($(\tikztostart)!1-\pv{pos}!(\tikztotarget)!\pv{height}!270:(\tikztotarget)$)
  .. (\tikztotarget)\tikztonodes}},
  settings/.code={\tikzset{quiver/.cd,#1}
      \def\pv##1{\pgfkeysvalueof{/tikz/quiver/##1}}},
  quiver/.cd,pos/.initial=0.35,height/.initial=0}

\tikzset{tail reversed/.code={\pgfsetarrowsstart{tikzcd to}}}
\tikzset{2tail/.code={\pgfsetarrowsstart{Implies[reversed]}}}
\tikzset{2tail reversed/.code={\pgfsetarrowsstart{Implies}}}
\tikzset{no body/.style={/tikz/dash pattern=on 0 off 1mm}}

\ProvideDocumentCommand{\hypersetup}{m}{}

\newtheorem{theoremm}{Theorem}[section]

\newtheorem{theoremmmm}{Theorem}

\declaretheorem[style=plain,name=Theorem,numberlike=theoremmmm]{theorematic}

\declaretheorem[style=plain,name=Theorem,numberlike=theoremm]{theorem}
\declaretheorem[style=plain,name=Theorem,numbered=no]{theorem*}

\declaretheorem[style=plain,name=Lemma,numberlike=theoremm]{lemma}
\declaretheorem[style=plain,name=Proposition,numberlike=theoremm]{proposition}
\declaretheorem[style=plain,name=Corollary,numberlike=theoremm]{corollary}

\declaretheorem[style=definition,name=Definition,numberlike=theorem]{definition}

\declaretheorem[style=remark,name=Example,numberlike=theorem]{example}
\declaretheorem[style=remark,name=Remark,numberlike=theorem]{remark}
\declaretheorem[style=remark,name=Notation,numberlike=theorem]{notation}

\usepackage[cal=boondoxo, calscaled=0.96, bb=dsserif]{mathalfa}

\newcommand{\on}[1]{\operatorname{#1}}
\newcommand{\setj}[1]{\left\{ #1 \right\}}

\newcommand*{\xiso}{%
\overset{\smash{\raisebox{-0.65ex}{\ensuremath{\scriptstyle\sim}}}%
                \,}%
        {\to}%
}

\DeclareFontFamily{U}{DSSerif}{\skewchar \font =45}
\DeclareFontShape{U}{DSSerif}{m}{n}{<-> s*[1]  DSSerif}{}
\DeclareFontSubstitution{U}{DSSerif}{m}{n}
\DeclareMathAlphabet{\mathbbbb}{U}{DSSerif}{m}{n}

\DeclareMathAlphabet\EuRoman{U}{eur}{m}{n}
\SetMathAlphabet\EuRoman{bold}{U}{eur}{b}{n}
\newcommand{\euler}{\EuRoman}

\newcommand{\projca}{\on{proj}_{\mathcal{C}}(A)}
\newcommand{\projda}{\on{proj}_{\mathcal{D}}(A)}
\newcommand{\modca}{\on{mod}_{\cat{C}}(A)}
\newcommand{\modda}{\on{mod}_{\cat{D}}(A)}
\newcommand{\bimod}{\on{bimod}_{\cat{C}}}
\newcommand{\bimodd}{\on{bimod}_{\cat{D}}}
\newcommand{\modw}[1]{\on{mod}_{\mathcal{C}}(#1)}
\newcommand{\modv}[1]{\on{mod}_{\mathcal{D}}(#1)}

\newcommand{\radca}{\on{Rad}_{A}^{\mathcal{C}}}

\newcommand{\cp}{\mathcal{C}_{p}}

\newcommand{\homa}{\on{Hom}_{\blank A}}

\newcommand{\kotimes}{\otimes_{\mathbb{k}}}

\newcommand{\gr}[1]{\left\llbracket #1 \right\rrbracket}

\usepackage[a4paper, margin=1in]{geometry}
\usepackage{lscape,mathdots}
\renewcommand{\phi}{\varphi}

\DeclareMathOperator{\rep}{rep}

\makeatletter
\def\Ddots{\mathinner{\mkern1mu\raise\p@
\vbox{\kern7\p@\hbox{.}}\mkern2mu
\raise4\p@\hbox{.}\mkern2mu\raise7\p@\hbox{.}\mkern1mu}}
\makeatother

\DeclareMathOperator{\vect}{\mathsf{vect}}
\DeclareMathOperator{\proj}{\on{proj}}

\DeclareMathOperator{\Irr}{Irr} 
\DeclareMathOperator{\FPdim}{FPdim}
\DeclareMathOperator{\Ext}{Ext}
\DeclareMathOperator{\iExt}{\underline{\Ext}}

\usepackage{quiver}
\usepackage{dsfont} 

\usepackage{enumitem}
\setenumerate[1]{label=(\roman*)} 

\title[Species and hereditary algebras]{Algebras in fusion categories: species and hereditary algebras}
\author{Edmund Heng}
\author{Mateusz Stroi{\' n}ski}

\begin{document}

\begin{abstract}
    We initiate the study of non-semisimple algebras in fusion categories by establishing the framework of $\cat{C}$-species -- analogous to the framework of species and quivers used in the study of Artin algebras.
    Under the (necessary) assumption that the fusion category is separable, we show that any algebra is Morita equivalent to an admissible quotient of the path algebra of a $\cat{C}$-species. Moreover, we show that an algebra is hereditary if and only if no further quotient is required. These results generalise that of Gabriel's for finite-dimensional algebras.
\end{abstract}

\maketitle

\tableofcontents

\addtocontents{toc}{\protect\setcounter{tocdepth}{1}}
\section{Introduction}
Throughout this paper, $\mathbb{k}$ is an algebraically closed field.
\subsection*{Motivation and results}
Algebras in fusion categories are essential in the study of finite module categories over fusion categories, as every finite module category over $\cat{C}$ is equivalent to the category of modules over some algebra $A \in \cat{C}$ \cite{Os, EGNO}.
Over the past decades, the field has mainly focused on semisimple algebras, i.e.\ algebras whose modules are all semisimple.
Beyond the semisimple algebras, our knowledge is surprisingly limited.
In fact, it is not only until the recent work by the second-named author and collaborators \cite{CSZ} (as a corollary in the fusion category setting) that we could discern semisimple algebras from non-semisimple ones through the following criterion: an algebra is semisimple if and only if it is semisimple as a module over itself (see \autoref{cor:reconcilesemisimplealg})  -- a fact which is standard for finite-dimensional algebras.

Nonetheless, there is a wealth of interesting actions of fusion categories on finite abelian categories which are not necessarily semisimple.
For example, a finite abelian category $\cat{M}$ equipped with a (strong) categorical action of a finite group $G$ is the same as $\cat{M}$ being a module category over the category $\vect_G$ of finite-dimensional $G$-graded vector spaces.
As such, $\cat{M}$ is equivalent to the category of $G$-graded modules over some $G$-graded algebra, which is semisimple if and only if $\cat{M}$ is.
On the other hand, the papers \cite{HengCox,EH} show that representation categories of quivers carry interesting actions of fusion categories, which include examples such as the Fibonacci fusion category $\Fib$ acting on $\rep(A_4)$ (bipartite) \cite[Example 3.2]{HengCox} -- in contrast, the only indecomposable semisimple category that $\Fib$ acts on is itself.

The goal of this paper is to initiate the study of all algebras in fusion categories by developing foundational results akin to those of finite-dimensional algebras (in fact, our results encompass the classical results by viewing finite-dimensional algebras as algebras in $\vect$).
We prove the following generalisation of Gabriel's theorem \cite{Ga1,Ga2}; for the definition of $\cat{C}$-species and its path algebra, refer to \autoref{defn:speciesandquiver} and \ref{pathalgebra}.
Recall that a fusion category is said to be \emph{separable} if its Drinfeld centre is semisimple (e.g.\ when $\on{char}(\mathbb{k}) = 0$).
\begin{theorematic}[= \autoref{extspecies} and \autoref{noidealpaths}] \label{introthm}
    Let $A \in \cat{C}$ be an algebra in a separable fusion category $\cat{C}$.
    There exists a $\cat{C}$-species Q and a two-sided ideal $I$ in its path algebra $T_Q$ such that $A$ is Morita equivalent to the algebra $T_Q/I$.
    Moreover, $I$ is trivial if and only if $A$ is hereditary, i.e.\ $\modca$ has global projective dimension at most 1.
\end{theorematic}
Note that the assumption of separability of $\cat{C}$ is necessary; see \autoref{rmk:SemisimpleCentreNecessary}.
Moreover, Gabriel's result is recovered by choosing $\cat{C} = \vect$, where its Drinfeld centre $\cat{Z(\vect)} \simeq \vect$ is always semisimple (regardless of the characteristic of $\mathbb{k}$).

\begin{remark}
Readers familiar with finite-dimensional algebras might wonder why we still require the notion of species when the underlying field is algebraically closed. The point is that even so, there may be more than one Morita class of semisimple algebras in $\cat{C}$. For example, Morita classes of semisimple algebras in $\vect_G$ are classified by subgroups $H \subseteq G$ and a class of $2$-cocycle of $H$; see \autoref{eg:Z2divisionalg}.
\end{remark}

When the semisimple part of $T_Q$ is a product of trivial algebras, the $\cat{C}$-species $Q$ is, by definition, a $\cat{C}$-quiver (see \autoref{defn:speciesandquiver}).
We also provide a criterion that determines when $A$ is Morita equivalent to the quotient of a path algebra of a $\cat{C}$-quiver.
For the following, recall that the Grothendieck group $\gr{\modca}$ is a $\mathbb{Z}_+$-module over the fusion ring $\gr{\cat{C}}$.
\begin{theorematic}[\autoref{quiversviaGr} and \ref{cor:fusionquivercondition}]
    Let $A$ be an algebra in a separable fusion category $\cat{C}$.
    Then an algebra $A$ is Morita equivalent to $T_Q/I$ with $Q$ a $\cat{C}$-quiver if and only if $\gr{\modca}$ is equivalent to $\gr{\cat{C}}^{\oplus n}$ as a $\mathbb{Z}_+$-module.
    In particular, if $A$ is also hereditary, then $A$ is Morita equivalent to $T_Q$ for $Q$ a $\cat{C}$-quiver.
\end{theorematic}
It then follows from the main theorem in \cite{EH} that a hereditary algebra that is Morita equivalent to $T_Q$ for $Q$ a $\cat{C}$-quiver has finitely many indecomposable modules (up to isomorphism) if and only if $Q$ is associated with Coxeter graphs that are disjoint unions of Coxeter--Dynkin diagrams (the diagrams which classify finite Coxeter groups).
In fact, one of the main motivations of this work is to expand the rich connection between Weyl groups and finite-dimensional algebras to all Coxeter groups via algebras in fusion categories; this will be explored in future works.
It is also an interesting problem to provide a finite-representation-type condition for all hereditary algebras (i.e.\ $\cat{C}$-species instead of $\cat{C}$-quivers), which we expect to be similar to Dlab--Ringel's generalisation of Gabriel's classification to species \cite{DR}.

\subsection*{Algebras in finite tensor categories}
While the theory of Jacobson radicals for algebra objects introduced in \cite{CSZ} applies to the general setting of finite tensor categories, the generalisation of Gabriel's theorem above (\autoref{introthm}) cannot hold in this level of generality; \autoref{eg:semisimplecentre} shows that it already fails in the setting of fusion categories that are not separable. 
Nevertheless, we develop the theory of $\cat{C}$-species in the setting of finite tensor categories as much as possible -- this culminates in a weak variant of Gabriel's theorem for algebra objects in finite tensor categories, which we prove in \autoref{weakGabriel}. 
In particular, in this generalisation, we need to require that the semisimple quotient $A/J$ of the given algebra $A$ is Morita equivalent to a separable algebra -- this is a condition which always holds in a separable fusion category, and which fails in \autoref{eg:semisimplecentre}. Additionally, separability in finite tensor categories is not preserved under Morita equivalence, as noted in \autoref{rmk:FTCseparable}. Beyond that, we further need a splitting of the projection $J \twoheadrightarrow J/J^{2}$, analogous to the classical case (\cite[Proposition~4.1.10]{Ben}) -- see \autoref{def:radicalsplit}.

We emphasise that even though this means that the structural result of algebras in the setting of non-semisimple finite tensor categories may only take on a weaker form, $\cat{C}$-quivers and $\cat{C}$-species still provide a way of systematically producing examples of well-behaved non-exact module categories.

\subsection*{Similar works}
Realising a module category via a $\cat{C}$-species can be thought of as describing it as a collection of exact module categories ``glued together'' by extensions. This idea bears some resemblance to the Jordan--H{\" o}lder theory of \cite{MM}. 
However, it is not the same -- while simple transitive $2$-representations over (the category of projectives of) a finite tensor category are (the categories of projectives of) the exact module categories, the action preorder (in the sense of \cite{MM}) of a finite module category over a finite tensor category is symmetric, and hence the associated Jordan--H{\" o}lder filtration is split. The difference is perhaps best explained by observing that the Jordan--H{\" o}lder theory of \cite{MM} is defined on the level of Grothendieck groups, while the arrows in the path algebra in our sense cannot be detected by the Grothendieck group at all: we find $\gr{\modw{A}} \simeq \gr{\modw{A/J}}$.

Despite having similar names, our notion of $\cat{C}$-quivers is different from the notion of fusion quivers in \cite{Sch}, which describes monoidal -- rather than module -- categories from quivers. On the other hand, the actions studied in \cite{Bet} could be encoded in terms of the $\cat{C}$-quivers and $\cat{C}$-species we consider, and it would be interesting to determine the quivers obtained in this manner.

\subsection*{Outline}
A brief outline of this paper is as follows.
In \autoref{prelims}, we collect preliminaries on finite tensor categories $\cat{C}$, including their module categories, and the notions of algebras, tensor algebras and division algebras in $\cat{C}$.
In \autoref{species}, we introduce the notion of $\cat{C}$-species and prove the weak variant of Gabriel's theorem as \autoref{weakGabriel}.
We focus on fusion categories $\cat{D}$ from \autoref{sec:fusioncat} onwards. Here we prove stronger results about radicals and semisimple algebras in this setting, and also define the notions of internal Ext and the directed graph associated to a finite $\cat{D}$-module category.
Our main theorems are proven in \autoref{mainresult}, where we work under the necessary assumption of $\cat{D}$ being a separable fusion category.

\subsection*{Acknowledgement}
The authors are deeply grateful to P.\ Etingof for putting us in contact, which led to the birth of this work.
We are also indebted to K.\ Coulembier for important discussions and for \autoref{eg:semisimplecentre} which shows the necessity of having a semisimple Drinfeld centre.
We would like to thank J.\ Külshammer for help with terminologies, and C.\ Lomp for pointing us to the work \cite{AMS}.
We also thank C.\ Schweigert, R.\ Kinser and B.\ \c{S}enol for helpful discussions.
E.H.\ is supported by the Australian Research Council grant DP230100654.
M.S.\ is supported by the Knut and Alice Wallenberg Foundation (Grant No.~2024.0339). Large parts of this work were done while M.S.\ was supported by Lundström-Åmans stipendiestiftelse.

\addtocontents{toc}{\protect\setcounter{tocdepth}{2}}
\section{Preliminaries}\label{prelims}

\subsection{Abelian categories}

Let $\mathbb{k}$ be an algebraically closed field. Throughout we assume all categories and functors considered to be $\mathbb{k}$-linear, unless otherwise stated. 

\begin{remark}
    It seems very likely that the results of this document hold equally well when $\mathbb{k}$ is a perfect, not necessarily algebraically closed field. However, this would require dealing with additional difficulties, e.g.\ in the proof of naturality of isomorphisms considered in \autoref{extarrowsextra}, and in \autoref{hereditary} we would need to adjust the definition of the Frobenius--Perron dimension, following \cite[Section~3]{Sa}.
\end{remark}

\begin{definition}
Given categories $\cat{A,B}$, we denote by $\cat{A} \kotimes \cat{B}$ their $\mathbb{k}$-linear tensor product. Its objects are pairs of objects in $\cat{A}$ and $\cat{B}$, and morphisms spaces are tensor products of those in $\cat{A}$ and $\cat{B}$. As such, $\mathbb{k}$-linear functors from $\cat{A}\otimes_{\mathbb{k}} \cat{B}$ are equivalently $\mathbb{k}$-bilinear functors from $\cat{A} \times \cat{B}$.    
\end{definition}

\begin{definition}
    Given additive categories $\cat{A,B}$, we denote by $\cat{A}\oplus \cat{B}$ the product category, whose objects are pairs $(X,Y)_{X \in \cat{A}, Y \in \cat{B}}$ and morphism spaces are direct sums of those in $\cat{A}$ and $\cat{B}$. Via the canonical projections, $\cat{A}\oplus \cat{B}$ is a bicategorical product in the $2$-category of additive categories, and via the injections $X \mapsto (X,0)$ and $Y \mapsto (0,Y)$ it simultaneously becomes a bicategorical coproduct in that $2$-category, justifying our use of the direct sum symbol.
\end{definition}

\begin{definition}
 Given a category $\cat{A}$, an {\it ideal $\mathfrak{I}$ in $\cat{A}$} consists of subspaces $\mathfrak{I}(X,Y) \subseteq \cat{A}(X,Y)$ for all $X,Y \in \cat{A}$, such that for all $g \in \mathfrak{I}$ and morphisms $h,f$ of $\cat{A}$ such that $h \circ g \circ f$ is defined, the composite $h \circ g \circ f$ lies in $\mathfrak{I}$. 
 
 The quotient category $\cat{A}/\mathfrak{I}$ is the category whose objects are given by $\on{Ob}(\cat{A}/\mathfrak{I})=\on{Ob}(\cat{A}) $, and whose morphism spaces are given by $\on{Hom}_{\cat{A}/\mathfrak{I}}(X,Y) = \on{Hom}_{\cat{A}}(X,Y)/\mathfrak{I}(X,Y)$, with the obvious induced composition. It satisfies the universal property that all functors from  $\cat{A}$ vanishing on $\mathfrak{I}$ factor through the canonical quotient functor $\cat{A} \to \cat{A}/\mathfrak{I}$.

 Given an ideal $\mathfrak{I}$ and a positive integer $n$, $\mathfrak{I}^{n}$ is the ideal generated by $n$-fold composites of morphisms in $\mathfrak{I}$. As such, we find that for a pair of positive integers $m,n$, such that $m>n$, we have inclusions $\mathfrak{I}^{n}(X,Y) \subseteq \mathfrak{I}^{m}(X,Y)$ for all $X,Y$. In other words, $\mathfrak{I}^{n}$ is a subideal of $\mathfrak{I}^{m}$.
 
 In that case, we define $\mathfrak{I}^{m}/\mathfrak{I}^{n}(X,Y) \defeq \mathfrak{I}^{m}(X,Y)/\mathfrak{I}^{n}(X,Y)$. In fact, $\mathfrak{I}^{m}/\mathfrak{I}^{n}$ is the quotient object in the category of {\it profunctors} on $\cat{A}$, but we will not need that fact.
\end{definition}

\begin{definition}
  Given an abelian category $\cat{A}$, we denote 
  \begin{itemize}
      \item $\mathbb{L}(\cat{A})$ as the full subcategory of $\cat{A}$ whose objects are the semisimple objects of $\cat{A}$;
      \item $\projectives{\cat{A}}$ as its category of projectives;
      \item $\on{Irr}(\cat{A)}$ as a chosen set of representatives of its isomomorphism classes of simple objects; and
      \item $\gr{\cat{A}}$ as the Grothendieck group of $\cat{A}$.
  \end{itemize}
  
\end{definition}

\begin{definition}
 A {\it finite abelian category} is an abelian ($\mathbb{k}$-linear) category $\cat{A}$ whose objects are of finite length, $\on{Hom}$-spaces are of finite dimension, and $\on{Irr}(\cat{A})$ is finite. Equivalently, $\cat{A} \simeq \on{mod}_{\mathbb{k}}(A)$, where $A$ is a finite-dimensional (associative, unital) $\mathbb{k}$-algebra and $\on{mod}_{\mathbb{k}}(A)$ is the category of its finite-dimensional modules.
\end{definition}

\begin{definition}
 Given an additive category $\cat{A}$, its {\it (Jacobson) radical} is the unique ideal in $\cat{A}$ such that for all $X \in \cat{A}$, the space $\mathfrak{J}(\mathcal{A})(X,X)$ is the Jacobson radical $J(\on{End}_{\cat{A}}(X))$ of $\on{End}_{\cat{A}}(X)$.
\end{definition}

\begin{definition}
   Let $\cat{A}$ be a finite abelian category and let $X \in \cat{A}$. The {\it projective cover} of $X$ is an epimorphism $p_{0}: P \twoheadrightarrow X$ such that any $\alpha \in \on{End}_{\cat{A}}(P)$ satisfying $p_{0}\circ \alpha = p_{0}$ is an automorphism of $P$. The projective cover is unique up to a unique isomorphism, and any object of a finite abelian category admits a projective cover.
\end{definition}

\begin{definition}\label{topsocle}
  Given a finite abelian category $\cat{A}$, we have an adjoint triple 
\[\begin{tikzcd}[column sep=large]
	{\mathbb{L}(\cat{A})} & {\cat{A}}
	\arrow[hook, from=1-1, to=1-2]
	\arrow["{\on{top}}"', shift right=3, from=1-2, to=1-1]
	\arrow["{\on{soc}}", shift left=3, from=1-2, to=1-1]
\end{tikzcd}\]
and furthermore, the composite $\mathbb{L}(\cat{A}) \hookrightarrow \cat{A} \twoheadrightarrow \cat{A}/\mathfrak{J}(\cat{A})$ is an equivalence of categories, under which we identify the two. 

The {\it radical} of an object $X \in \cat{A}$ is the kernel of the unit of adjunction $\eta_{X}: X \rightarrow \iota(\on{top}(M))$.
Under an equivalence $\mathcal{A} \simeq \on{mod}_{\mathbb{k}}A$, we find $\mathcal{A}/\mathfrak{J}(\mathcal{A}) \simeq \on{mod}_{\mathbb{k}}(A/J(A))$, and the adjunction $\on{top} \dashv \iota$ is that between induction and restriction along the projection $\pi: A \twoheadrightarrow A/J(A)$. Since we assume $\mathbb{k}$ to be algebraically closed, in particular perfect, we find a section $\sigma_{0}: A/J \rightarrow A$, and a further adjoint triple associated to the restriction functor $\on{res}(\sigma_{0})$, with $P_0$ and $I_0$ correspond to taking projective covers and injective envelops respectively: 
\[\begin{tikzcd}[ampersand replacement=\&,column sep=large]
	{\cat{A}} \&\& {\mathbb{L}(\cat{A})}
	\arrow["{\on{res}(\sigma_{0})}"{description}, from=1-1, to=1-3]
	\arrow["{I_{0}}", shift left=3, from=1-3, to=1-1]
	\arrow["{P_{0}}"', shift right=4, from=1-3, to=1-1]
\end{tikzcd}\]
Since the composite $\on{res}(\sigma_{0}) \circ \on{res}(\pi)$ is naturally isomorphic to the identity functor, by passing to left adjoints we find
  \begin{equation}\label{topcoverbijection}
      \on{top}(P_{0}(X)) \simeq \on{top}(X) \text{ for all }X,\text{ and } P_{0}(\on{top}(P)) \simeq P \text{ for all projective }P.
  \end{equation} 
\end{definition}

\begin{definition}
 The {\it $\on{Ind}$-completion} $\on{Ind}(\cat{A})$ of a category $\cat{A}$ is defined as the full subcategory of the presheaf category on $\cat{A}$ consisting of filtered colimits of representable objects. As such, there is a full, faithful embedding of $\cat{A}$ in $\on{Ind}(\cat{A})$, which is a corestriction of the Yoneda embedding. 
 
 The $\on{Ind}$-completion of a finite abelian category is a complete, cocomplete abelian category, and the embedding of $\cat{A}$ in it is an exact functor.
\end{definition}

\begin{proposition}\label{unnaturality}
 Let $\cat{A}$ be a finite semisimple category and let $\cat{B}$ be an additive category. Then two functors $\functor{F},\functor{G}: \cat{A} \rightarrow \cat{B}$ are naturally isomorphic if and only if $\functor{F}(L) \simeq \functor{G}(L)$, for any $L \in \on{Irr}(\cat{A})$
\end{proposition}

\begin{proof}
  Clearly, $\cat{A}$ is the completion $\on{Irr}(\cat{A})^{\oplus}$ of $\on{Irr}(\cat{A})$ under finite direct sums, where we view $\on{Irr}(\cat{A})$ as a ($\mathbb{k}$-linear) category obtained from the full subcategory of $\cat{A}$ with the same objects as $\on{Irr}(\cat{A})$. The claim follows by combining the isomorphism $\on{Fun}_{\mathbb{k}}(\cat{E},\cat{B}) \simeq \on{Fun}_{\mathbb{k}}(\cat{E}^{\oplus},\cat{B})$, with the obvious isomorphism $\on{Fun}_{\mathbb{k}}(\on{Irr}(\cat{A}), \cat{B}) \simeq \cat{B}^{|\on{Irr}(\cat{A})|}$.
\end{proof}
 
\subsection{Tensor and fusion categories}
We follow \cite{EGNO} for the following two definitions.
\begin{definition}
 A \emph{finite tensor category} is a finite abelian rigid monoidal category $(\cat{C},\otimes,\mathbb{1},\mathsf{a})$ satisfying $\on{End}_{\cat{C}}(\mathbb{1}) \simeq \mathbb{k}$.
\end{definition}

\begin{definition}
 A {\it fusion category} is a finite tensor category which is semisimple. 
\end{definition}

For the remainder of this document, we fix a finite tensor category $\cat{C}$ (specifying semisimplicity or fusion whenever required). From \autoref{sec:fusioncat} onwards, we will also work with a fusion category $\cat{D}$.

\begin{definition}
We remark that while we use the same convention as \cite{EGNO} regarding terminology for left and right duals in a rigid monoidal category, we use the opposite convention regarding the associated decorations. Namely, we will denote by $\rightdual{V}$ the {\it right dual} of an object $V \in \cat{C}$, equipped with morphisms $\coev\from \mathbb{1}\to \rightdual{V} \otimes V$ and $\ev\from V\otimes \rightdual{V} \to \mathbb{1}$; and vice versa.
\end{definition}

\subsection{Module categories}
For the rest of this section, $\cat{C}$ denotes a finite tensor category.
\begin{definition}
For a monoidal category $\cat{C}$, a {\it left module category over }$\cat{C}$ is a category $\cat{M}$ equipped with an action functor $\lact\from \cat{C}\kotimes \cat{M} \to \cat{M}$ sending $(V,X)$ to $V \lact X$, equipped with coherent isomorphisms $(V\otimes W)\lact X \xiso V\lact (W\lact X)$ and $\mathbb{1}\lact X \xiso X$; see \cite[Definition 7.1.2]{EGNO} for details.
A right module category over $\cat{C}$ is defined in a similar fashion.
\end{definition}

\begin{definition}
 Given $\cat{C}$-module categories $\cat{M,N}$, a $\cat{C}$-module functor from $\cat{M}$ to $\cat{N}$ is a functor $\functor{F}\from \cat{M} \to \cat{N}$, together with coherent isomorphisms $V\lact \functor{F}(X) \xiso \functor{F}(V\lact X)$; see \cite[Definition 7.1.2]{EGNO} for details.
\end{definition}

Combining \cite[Remark~4]{Os} with \cite[Theorem~1.2]{Kel}, we find the following:
\begin{lemma} \label{adjointmodule}
  An adjoint (left or right) of a $\cat{C}$-module functor itself is a $\cat{C}$-module functor.
\end{lemma}

\begin{definition}
Given an object $X \in \cat{M}$, the {\it internal $\on{Hom}$ from $X$} is the functor $\hom{X}{-}$ right adjoint to the functor $- \lact X\from \cat{C}\to \cat{M}$. For $Y \in \cat{M}$, the {\it internal $\on{Hom}$ from $X$ to $Y$} is the object $\hom{X}{Y} \in \cat{C}$.
We also denote $\underline{\on{End}}_{\cat{M}}(X) \coloneqq \hom{X}{X}$.

We say that a $\cat{C}$-module category $\cat{M}$ is {\it closed} if $\hom{X}{-}$ exists for all $X \in \cat{M}$. 
\end{definition}

The following claim is established in \cite[\S1]{Kock} in the symmetric case, but both this result as well as the proofs of \cite{Kock} carry over verbatim also to the case of non-symmetric $\cat{C}$, where the notion of enrichment in $\cat{C}$ is that of enrichment in a bicategory of \cite{Street}, applied to the one-object bicategory $\mathbf{B}\cat{C}$ obtained by delooping $\cat{C}$.

\begin{proposition}\label{EnrichmentsKock}
    Let $\cat{M,N}$ be closed $\cat{C}$-module categories. The assignments $\on{Ob}(\underline{\cat{M}}) := \on{Ob}(\cat{M})$ and $\on{Hom}_{\underline{\cat{M}}}(X,Y) := \underline{\on{Hom}}_{\cat{M}}(X,Y)$ define a $\cat{C}$-enriched category $\underline{\cat{M}}$. Given a $\cat{C}$-module functor $F\from \cat{M} \to \cat{N}$, 
    the composite 
\begin{equation}\label{internaldiagram}
\begin{tikzcd}
	{\on{Hom}_{\mathcal{C}}(V, \underline{\on{Hom}}_{\mathcal{M}}(X,Y))} && {\on{Hom}_{\mathcal{C}}(V, \underline{\on{Hom}}_{\mathcal{N}}(F(X),F(Y)))} \\
	{\on{Hom}_{\mathcal{M}}(V \triangleright X, Y)} & {\on{Hom}_{\mathcal{N}}(F(V \triangleright X), F(Y))} & {\on{Hom}_{\mathcal{N}}(V \triangleright F(X), F(Y))}
	\arrow["\simeq", from=1-1, to=2-1]
	\arrow["{F_{V\triangleright X, Y}}", from=2-1, to=2-2]
	\arrow["{-\circ \varphi_{V,F(X)}}", from=2-2, to=2-3]
	\arrow["\simeq", from=2-3, to=1-3]
\end{tikzcd}
\end{equation}
induces, by Yoneda lemma,
a morphism $\underline{F_{X,Y}}: \underline{\on{Hom}}_{\mathcal{M}}(X,Y) \rightarrow \underline{\on{Hom}}_{\mathcal{N}}(F(X),F(Y))$. The assignments $\on{Ob}(\underline{F}) = \on{Ob}(F)$ and $\underline{F}_{X,Y} := \underline{F_{X,Y}}$ define a $\cat{C}$-enriched functor $\underline{F}\from \underline{\cat{M}} \to \underline{\cat{N}}$.

In particular, $\underline{F}_{X,X}$ is a morphism of algebra objects in $\cat{C}$ (see \autoref{defn:algebra}).
\end{proposition}

\begin{lemma}\label{enrichedhom}
    Let $\cat{M,N}$ be closed $\cat{C}$-module categories. Then $F$ is full and faithful if and only if $\underline{F}_{X,Y}$ is invertible for all $X,Y$.
\end{lemma}

\begin{proof}
 If $F$ is full and faithful, then the morphisms in \autoref{internaldiagram} are invertible, and hence so is $\underline{F}_{X,Y}$, since the Yoneda embedding reflects isomorphisms. If $\underline{F}_{X,Y}$ is invertible, then so are the morphisms in \autoref{internaldiagram}, and hence so is $F_{X,Y}$, since it can be obtained by setting $V = \mathbb{1}$ in \autoref{internaldiagram}.
\end{proof}

\begin{definition}
  A {\it finite $\cat{C}$-module category} is a module category over $\cat{C}$ which is finite abelian and such that for every $X \in \cat{M}$, the functor $- \lact X$ is right exact.

 In particular, the internal Hom functor $\hom{X}{-}$, defined as the right adjoint of $-\lact X$, exists for all $X \in \cat{M}$. In other words, a finite $\cat{C}$-module category is closed.
\end{definition}
\begin{remark}
    The standard definition of a finite $\cat{C}$-module category (cf.~ \cite[Definition 7.3.1]{EGNO}) assumes the stronger condition that the functor $- \lact X$ is exact. 
    Our weaker condition is equivalent to this; see \autoref{cor:exactonboth}.
\end{remark}

We note that when $\cat{C}$ is semisimple (i.e.\ fusion), the internal $\on{Hom}$ is easy to compute.
Being the right adjoint of $- \lact X$ by definition, we have $\on{Hom}_\cat{M}(L\lact X, Y) \simeq \on{Hom}_\cat{C}(L, \hom{X}{Y})$.
By going through all of the simple objects in $\cat{C}$, we see that
\begin{equation} \label{eqn:internalhomformula}
\hom{X}{Y} \simeq \bigoplus_{L \in \on{Irr}(\cat{C})} L^{\oplus d_L}
\end{equation}
where $d_L \coloneqq \dim_{\mathbb{k}} \on{Hom}_\cat{M}(L\lact X, Y) = \dim_{\mathbb{k}} \on{Hom}_\cat{C}(L, \hom{X}{Y})$.

\begin{definition}
 We say that an object $X \in \cat{M}$ is {\it $\cat{C}$-projective} if the internal Hom functor $\hom{X}{-}$ is exact.   

  A $\cat{C}$-projective object is said to be a $\cat{C}$-projective generator if $\hom{X}{Y} = 0$ implies $Y = 0$.
  
  Note that if $\cat{C}$ is a fusion category, then an object is $\cat{C}$-projective if and only if it is projective.
\end{definition}

The following is observed already in \cite[Corollary~3.15]{EO}; see \cite[Proposition~4.3, Proposition~4.5]{SZ} for a detailed proof in a slightly more general setting.

\begin{proposition}\label{cprojectiveobjects}
  An object $X$ in a finite $\cat{C}$-module category $\cat{M}$ is $\cat{C}$-projective if and only if for 
 any projective object $P \in \cat{C}$, the object $P \lact X$ is projective, and it is a $\cat{C}$-projective generator if and only if any projective object of $\cat{M}$ is a direct summand of one of the form $P \lact X$.
\end{proposition}

In view of \autoref{cprojectiveobjects}, the following is a rephrasing of \cite[Definition~3.1]{EO}.

\begin{definition}
  A finite $\cat{C}$-module category $\cat{M}$ is said to be {\it exact} if all of its objects are $\cat{C}$-projective. 
\end{definition}

The following lemma is immediate.
\begin{lemma} \label{lem:semisimpleisexact}
 A semisimple $\cat{C}$-module category is automatically exact, since all of its objects are projective. If $\cat{C}$ is fusion, then $\mathbb{1}$ is projective, and hence any object of an exact $\cat{C}$-module category is projective, showing that an exact module category over a fusion category is necessarily semisimple.   
\end{lemma}

\subsection{Module categories are categories of modules}

\begin{definition}\label{defn:algebra}
 An {\it algebra object} in $\cat{C}$ consists of an object $A \in \cat{C}$ together with a multiplication map $\mu\from A \otimes A \to A$ and a unit map $\eta\from \mathbb{1} \to A$, satisfying associativity and unitality axioms of \cite[Definition~7.8.1]{EGNO}. Similarly to \cite{EGNO}, we will refer to algebra objects in $\cat{C}$ simply as {\it algebras in $\cat{C}$}.   
\end{definition}

\begin{definition}
 There is an obvious notion of a morphism of algebras, and we denote the category of algebras in $\cat{C}$ together with such morphisms by $\Alg{\cat{C}}$.   
\end{definition}

\begin{definition}
 Given $A \in \Alg{\cat{C}}$, a {\it right $A$-module (in $\cat{C}$)} is an object $M \in \cat{C}$ together with an action map $\nabla_{M}\from M \otimes A \to M$, satisfying multiplicativity and unitality axioms of \cite[Definition~7.8.5]{EGNO}. Similarly one defines left $A$-modules.

\end{definition}

\begin{definition}
  There is an obvious notion of morphism of right $A$-modules. We denote the category of right $A$-modules simply by $\modca$.
  
 The category $\modca$ of \emph{right} $A$-modules is canonically a \emph{left} $\cat{C}$-module category, which in the strict case is defined simply by $V \lact (M,\nabla_{M}) \defeq (V \otimes M, V \otimes \nabla_{M})$.
\end{definition}

The following is a slight reformulation of \cite[Theorem~7.10.1]{EGNO}:
\begin{theorem}[{\cite[Theorem~7.10.1]{EGNO}}]\label{EGNOreconstruction}
   If $P$ is a $\cat{C}$-projective generator for a finite $\cat{C}$-module category $\cat{M}$, then the functor $\hom{P}{-}\from \cat{M} \to \modw{\underline{\on{End}}_\cat{M}(P)}$ is an equivalence of $\cat{C}$-module categories.

   Furthermore, a projective generator for $\cat{M}$ is a fortiori a $\cat{C}$-projective generator for $\cat{M}$. Hence, for any finite $\cat{C}$-module category, there is an $A \in \Alg{\cat{C}}$ such that $\cat{M} \simeq \modca$.
\end{theorem}

\begin{corollary} \label{cor:exactonboth}
  For a finite $\cat{C}$-module category $\cat{M}$, the functor $- \lact -$ is exact in both variables.   
\end{corollary}

\begin{example}
    The forgetful functor sending $\vect_G \to \vect$ makes $\vect$ a $\vect_G$-module category; equivalently, $\vect$ is equipped with a (strong) categorical $G$-action with $g \lact V = V$ for all $g \in G$ and $V \in \vect$.
    Then $\vect \simeq \on{mod}_{\vect_G}(\underline{\on{End}}_{\cat{M}}(\mathbb{k}))$, where using \autoref{eqn:internalhomformula}, one shows that $\underline{\on{End}}_{\cat{M}}(\mathbb{k}) = \mathbb{k}[G]$; here, $\mathbb{k}[G]$ is viewed as a $G$-graded algebra with each $g \in G$ sitting in degree $g$. This result is a special case of the fundamental theorem of Hopf modules, applied to the Hopf algebra given by the dual to the group algebra.
\end{example}

\subsection{Bimodules, ideals and the \texorpdfstring{$\cat{C}$}{C}-radical of an algebra}

\begin{definition}[{\cite[Definition~7.8.25]{EGNO}}]
 Given $A,B \in \Alg{\cat{C}}$, an {\it $A$-$B$-bimodule object in $\cat{C}$} is an object $M \in \cat{C}$ which is a right $B$-module and left $A$-module, such that the two actions commute.

 A morphism of bimodules is a morphism of $\cat{C}$ which is simultaneously a morphism of $A$-modules and of $B$-modules. We denote the category of $A$-$B$-bimodules by $\bimod(A,B)$.
\end{definition}

\begin{definition}
 Given an $A$-$B$-bimodule $M$, the $\cat{C}$-module functor $- \otimes_{A} M\from \modca \to \modw{B}$ is defined on objects by $N \otimes_{A} M = \on{coeq}(N \otimes \nabla_{M}, \nabla_{N}\otimes M)$. The right $B$-action is inherited from that on $M$, and well-defined by associativity and (right) exactness of $\otimes$. It is extended to morphisms by functoriality of colimits. It becomes a $\cat{C}$-module functor as a consequence of associativity of $\otimes$.

 We say that $A,B \in \Alg{\cat{C}}$ are 
 \emph{Morita equivalent} if there is a $\cat{C}$-module equivalence between $\modca$ and $\modw{B}$. As a direct consequence of \cite[Proposition~7.11.1]{EGNO}, equivalently there are bimodules ${}_{A}M_{B}, {}_{B}N_{A}$ with bimodule isomorphisms $A \simeq M \otimes_{B} N$ and $B \simeq N \otimes_{A} M$.

 Observe that if we instead consider algebras in $\on{Ind}(\cat{C})$, then also the bimodules realizing the Morita equivalence need to be bimodules in $\on{Ind}(\cat{C})$, rather than just in $\cat{C}$.
\end{definition}

\begin{definition}
 A morphism $\varphi: A \rightarrow B$ in $\Alg{\cat{C}}$ endows $B$ with left and right actions of $A$, which commute with the regular actions of $B$ on itself. The functor $-\otimes_{A} B\from \modca \to \modw{B}$, which we denote by $\on{ind}(\varphi)$ and refer to as {\it induction along $\varphi$}, has an exact right adjoint $\on{res}(\varphi)$, which takes a $B$-module $M$ and sends it to the $A$-module whose $A$-action is given by $M \otimes A \xrightarrow{M \otimes \varphi} M \otimes B \xrightarrow{\nabla_{M}} M$. 
\end{definition}

\begin{definition}
 For $A \in \Alg{\cat{C}}$, a {\it (two-sided) ideal} in $A$ is a subobject of $A$ in $\bimod(A,A)$.
\end{definition}

\begin{definition}\label{idealtimesbimodule}
 Given $M \in \bimod(B,A)$ and an ideal $I$ in $A$, we denote by $MI$ the image of the morphism 
 \[
 M \otimes I \hookrightarrow M \otimes A \xrightarrow{\nabla_{M}} M.
 \]
 As a consequence of associativity and (right) exactness of $\otimes$, the object $MI$ is a $B$-$A$-subbimodule of $M$.
\end{definition}

\begin{proposition}
  Given an ideal $I$ in $A$, the quotient $A/I$ is again an algebra in $\cat{C}$, and the morphism $\pi\from A \to A/I$ is a morphism in $\Alg{\cat{C}}$. The functor $\on{res}(\pi)$ is full, faithful, and its essential image can be characterized as consisting of $M \in \modca$ satisfying $MI = 0$.
\end{proposition}

\begin{definition}
   Binary (more generally, finite) products in the category $\Alg{\cat{C}}$ of algebras in $\mathcal{C}$ are lifted from $\mathcal{C}$: given $A_{1},A_{2} \in \Alg{\cat{C}}$, the product algebra $A_{1} \times A_{2} \in \Alg{\cat{C}}$ is given by $\left(A_{1} \oplus A_{2}, \left(\begin{smallmatrix} \mu_{1} & 0 \\ 0 & \mu_{2}\end{smallmatrix}\right), \left(\begin{smallmatrix} \eta_{1} \\ \eta_{2} \end{smallmatrix} \right)\right)$. The canonical injections $\iota_{1}\from A_{1} \to A_{1} \oplus A_{2}$ and $\iota_{2}$ into the coproduct $A_{1} \oplus A_{2}$ in $\mathcal{C}$ realize $A_{1}$ and $A_{2}$ as ideals in $A_{1} \times A_{2}$, hence also non-unital subalgebras of $A_{1} \times A_{2}$. We denote the idempotent endomorphisms $\iota_{1} \circ \pi_{1}$ and $\iota_{2} \circ \pi_{2}$ by $e_{1}$ and $e_{2}$. 
\end{definition}

\begin{proposition}\label{modproduct}
    There is an equivalence $\modw{A_{1}\times A_{2}} \simeq \modw{A_{1}}\oplus \modw{A_{2}}$. 
\end{proposition}

\begin{proof}
    To an object $(M_{1},M_{2}) \in \modw{A_{1}}\oplus \modw{A_{2}}$, where hence $M_{i} \in \modw{A_{i}}$, we associate the object $M_{1}\oplus M_{2}$ in $\modw{A_{1}\times A_{2}}$, with the obvious $A_{1}\times A_{2}$-module structure. Conversely, we send $N \in \modw{A_{1}\times A_{2}}$ to the pair $(\varepsilon_{1}(N),\varepsilon_{2}(N))$, where $\varepsilon_{i}$ is the endomorphism
    \[
    N \xiso \mathbb{1}\otimes N \xrightarrow{\eta_{A_{1}\times A_{2}}} (A_{1}\times A_{2})\otimes N \xrightarrow{e_{i}\otimes N} (A_{1}\times A_{2})\otimes N \xrightarrow{\nabla_{N}} N.
    \]
    It is easy to extend these assignments to functors, and to show that $\varepsilon_{i}(M_{i}) = M_{i}$ and $\varepsilon_{i}(M_{j}) = 0$ if $i \neq j$. It remains to show that $N \simeq \varepsilon_{1}(N) \oplus \varepsilon_{2}(N)$.

    Since $\nabla_{N}$ is an epimorphism, $N = \on{Im}(\varepsilon_{1}N) + \on{Im}(\varepsilon_{2}N)$. To prove that the sum is direct, we show that $\varepsilon_{2}\circ \varepsilon_{1} =0$, as by symmetry we will find $\varepsilon_{1}\circ \varepsilon_{2} =0$. Let $A := A_{1}\times A_{2}$. The commutativity of all the faces of the following diagram follows from the interchange law, except for the bottom-right face, which follows from multiplicativity of action on $N$:
\[\begin{tikzcd}
	N & {\mathbb{1}\otimes N} & {A \otimes N} & {A \otimes N} & N \\
	{\mathbb{1}\otimes N} & {\mathbb{1}\otimes \mathbb{1}\otimes N} & {\mathbb{1}\otimes A\otimes N} & {\mathbb{1}\otimes A\otimes N} & {\mathbb{1}\otimes N} \\
	{A\otimes N} & {A\otimes \mathbb{1}\otimes N} & {A\otimes A\otimes N} & {A\otimes A\otimes N} & {A \otimes N} \\
	{A\otimes N} & {A\otimes \mathbb{1}\otimes N} & {A\otimes A\otimes N} & {A\otimes A\otimes N} & {A \otimes N} \\
	&&& {A \otimes N} & N
	\arrow["\simeq", from=1-1, to=1-2]
	\arrow["\simeq"', from=1-1, to=2-1]
	\arrow["{\eta_{A}\otimes N}", from=1-2, to=1-3]
	\arrow["\simeq"', from=1-2, to=2-2]
	\arrow["{e_{2}\otimes N}", from=1-3, to=1-4]
	\arrow["\simeq"', from=1-3, to=2-3]
	\arrow["{\nabla_{N}}", from=1-4, to=1-5]
	\arrow["\simeq"', from=1-4, to=2-4]
	\arrow["\simeq"', from=1-5, to=2-5]
	\arrow["\simeq", from=2-1, to=2-2]
	\arrow["{\eta_{A}\otimes N}"', from=2-1, to=3-1]
	\arrow["{\mathbb{1}\otimes \eta_{A}\otimes N}", from=2-2, to=2-3]
	\arrow["{\eta_{A}\otimes \mathbb{1}\otimes N}"', from=2-2, to=3-2]
	\arrow["{\mathbb{1}\otimes e_{2}\otimes N}", from=2-3, to=2-4]
	\arrow["{\eta_{A}\otimes A\otimes N}"', from=2-3, to=3-3]
	\arrow["{\mathbb{1}\otimes \nabla_{N}}", from=2-4, to=2-5]
	\arrow["{\eta_{A}\otimes A\otimes N}"', from=2-4, to=3-4]
	\arrow["{\eta_{A}\otimes N}"', from=2-5, to=3-5]
	\arrow["\simeq", from=3-1, to=3-2]
	\arrow["{e_{1}\otimes N}"', from=3-1, to=4-1]
	\arrow["{A\otimes \eta_{A}\otimes N}", from=3-2, to=3-3]
	\arrow["{e_{1}\otimes \mathbb{1}\otimes  N}"', from=3-2, to=4-2]
	\arrow["{A\otimes e_{2}\otimes N}", from=3-3, to=3-4]
	\arrow["{e_{1}\otimes A \otimes  N}"', from=3-3, to=4-3]
	\arrow["{A\otimes \nabla_{N}}", from=3-4, to=3-5]
	\arrow["{e_{1}\otimes A \otimes  N}"', from=3-4, to=4-4]
	\arrow["{e_{1}\otimes N}"', from=3-5, to=4-5]
	\arrow["\simeq", from=4-1, to=4-2]
	\arrow["{A\otimes \eta_{A}\otimes N}", from=4-2, to=4-3]
	\arrow["{A\otimes e_{2}\otimes N}", from=4-3, to=4-4]
	\arrow["{A\otimes \nabla_{N}}", from=4-4, to=4-5]
	\arrow["{\mu_{A}\otimes N}"', from=4-4, to=5-4]
	\arrow["{\nabla_{N}}"', from=4-5, to=5-5]
	\arrow["{\nabla_{N}}", from=5-4, to=5-5]
\end{tikzcd}\] 
This shows that $\varepsilon_{2} \circ \varepsilon_{1}$ equals the composite $\nabla_{N} \circ (\mu_{A} \otimes N) \circ (e_{1}\otimes e_{2}\otimes N) \circ (\eta_{A} \otimes \eta_{A} \otimes N)$. However, by definition of $\mu_{A}$, we have $\mu_{A} \circ (e_{1} \otimes e_{2}) = 0$, establishing the claim. 
\end{proof}

\begin{notation}
  For the remainder of this document, $A$ will always denote an algebra in a finite tensor category. Furthermore, given an algebra $A\in \cat{C}$, we introduce the following shorthand notations:
  \begin{itemize}
      \item $\projca$ is the category of projective objects in $\modca$.
      \item $\on{Irr}(A)$ is the set of (representatives of isomorphism classes of) simple objects in $\modca$.
      \item $\mathbb{L}(A)$ is the category of semisimple objects in $\modca$.
      \item $\homa(M,N)$ denotes $\on{Hom}_{\modca}(M,N)$.
  \end{itemize}
\end{notation}

\begin{definition}
  An algebra $A \in \cat{C}$ is said to be {\it simple} if it has no proper (not zero or equal to $A$ itself) two-sided ideals.   
\end{definition}

\begin{definition}\label{defn:semisimplealgebra}
  An algebra $A \in \cat{C}$ is said to be {\it semisimple} if it is isomorphic to a finite product of simple algebras.
\end{definition}

\begin{remark}
    Suppose $\cat{C}$ is moreover a fusion category. A common definition of a semisimple algebra $A$ is that $\modca$ is a semisimple category; see e.g.\ \cite[Definition 9]{Os}. These two definitions are equivalent -- refer to \autoref{CSZEO} and \autoref{cor:reconcilesemisimplealg}.
\end{remark}

\begin{definition}
 Given an ideal $I$ in $A$ and a positive integer $k$, we define $I^{k}$ inductively as $I^{k-1}I$, following \autoref{idealtimesbimodule}. An ideal $I$ is said to be {\it nilpotent} if there is $m$ such that $I^{m} = 0$.
\end{definition}

\begin{definition}
 A {\it $\cp$-stable ideal} in $\projca$ is an ideal $\mathfrak{I}$ in the additive category $\projca$ such that for $f \in \mathfrak{I}$ and $P \in \cat{C}$ projective, the morphism $P \lact f$ again belongs to $\mathfrak{I}$.
\end{definition}

\begin{definition}
 A {\it mixed subfunctor} of the functor $\on{Hom}_{\cat{C}}(-,A): \projectives{\cat{C}}\to \vect$ is a subobject $\mathtt{I}$ of $\on{Hom}_{\cat{C}}(-,A)$ in the category of such functors, such that for $P,Q\in \projectives{\cat{C}}$ and morphisms $g \in \mathtt{I}(P)$ and $f \in \homa(P \lact A, Q\lact A)$, the morphism
 \[
 Q \xiso Q \otimes \mathbb{1} \xrightarrow{Q \otimes \eta} Q \otimes A \xrightarrow{f} P\otimes A \xrightarrow{g \otimes A} A \otimes A \xrightarrow{\mu} A
 \]
 lies in $\mathtt{I}(Q)$, and such that the analogous condition for left $A$-module morphisms also holds.
\end{definition}

\begin{definition}
   The $\cp$-stable ideal $\mathfrak{J}_{\cat{C}}(A)$ (denoted by $\radca$ in \cite{CSZ}) is uniquely defined by
 \begin{equation}\label{eqmoduleradical}
 \mathfrak{J}_{\cat{C}}(A)(P\lact A,P'\lact A) =
 \left\{
   f \in \homa (P\triangleright A,P' \triangleright A)
 \ \Bigg|\ %
   \begin{aligned}
     &Q\triangleright f \in \mathfrak{J}(\projca)(Q\otimes P\triangleright A,Q\otimes P'\triangleright A) \\
     &\text{for all } Q \in \projectives{\cat{C}}
   \end{aligned}
 \right\}.
 \end{equation}
 for $P,P' \in \projectives{\cat{C}}$.
\end{definition}

\begin{theorem}\label{CSZBijection}
Fix an epimorphism $q: Q \twoheadrightarrow \mathbb{1}$, with $Q$ projective. By \cite[Section~5]{CSZ}, for a $\cp$-stable ideal $\mathfrak{I}$ in $\projca$, one obtains a mixed subfunctor $\mathbb{S}(\mathfrak{I})$ of $\on{Hom}_{\cat{C}}(-,A)$, by defining $\mathbb{S}(\mathfrak{I})(P)$ as the space of morphisms of the form
\[
P \xiso P\otimes \mathbb{1} \xrightarrow{P \otimes \eta} P \otimes A \xrightarrow{g} Q \otimes A \xrightarrow{q \otimes A} \mathbb{1} \otimes A \xiso A,
\]
where $g \in \mathfrak{I}(P\otimes A, Q \otimes A)$.
The mixed subfunctor $\mathbb{S}(\mathfrak{I})$ defined above is independent of the choice of $Q$ and $q: Q \twoheadrightarrow \mathbb{1}$.

Furthermore, the object $I$ representing the right exact extension of a mixed subfunctor $\mathtt{I}$ of $\on{Hom}_{\cat{C}}(-,A)$ is an ideal in $A$. Furthermore, the maps
 \[
    \left\{ \mathcal{C}_p\text{-stable ideals in } \projca \right\}
    \xrightarrow{\mathbb{S}} 
    \left\{ \text{mixed subfunctors of} \on{Hom}_{\cat{C}}(-,A) \right\} \to
    \left\{ \text{two-sided ideals of } A \right\} 
 \]
 are poset isomorphisms which respect multiplications. 
\end{theorem}

\begin{definition}[\protect{\cite[Definition~6.9]{CSZ}}]\label{internalradical}
 The Jacobson radical $J(A)$ of $A$ is defined as the ideal in $A$ corresponding to $\mathfrak{J}_{\cat{C}}(A)$ under the bijection of \autoref{CSZBijection}.
\end{definition}

\begin{theorem}[{\cite[Theorem~7.1]{CSZ}}]\label{CSZEO}
    The radical $J(A)$ is the greatest nilpotent ideal in $A$. The following conditions are equivalent:
   \begin{enumerate}
       \item $\modca$ is an exact $\cat{C}$-module category.
       \item $A$ is semisimple.
       \item $A$ has no non-trivial nilpotent ideals.
       \item $J(A) = 0$.
   \end{enumerate}
\end{theorem}

Observe that the proof of the following statement does not use the exactness of the monoidal structure in $\cat{C}$.
\begin{lemma}\label{ressimple}
 Let $A \in \Alg{\on{Ind}(\cat{C})}$, let $I$ be a two-sided ideal in $A$ and let $\pi: A \rightarrow A/I$ be the associated projection. Then for the restriction functor $\on{res}(\pi): \on{mod}_{\on{Ind}(\cat{C})}(A/I) \to \on{mod}_{\on{Ind}(\cat{C})}(A)$ and $L \in \Irr(\modw{A/I})$, we have $\on{res}(\pi)(L) \in \Irr(\modca)$. In other words, the restriction of a simple $A/I$-module along $\pi$ is a simple $A$-module.
\end{lemma}

\begin{proof}
 Observe first that in an abelian category, for a commutative square
\[\begin{tikzcd}[ampersand replacement=\&,cramped,sep=small]
	M \& {M'} \\
	N \& {N'}
	\arrow["f", from=1-1, to=1-2]
	\arrow["h", from=1-1, to=2-1]
	\arrow["{h'}", from=1-2, to=2-2]
	\arrow["g"', from=2-1, to=2-2]
\end{tikzcd}\]
the map $\widetilde{h'}: \on{Im}(f) \to \on{Im}(g)$ is mono if $h'$ is mono. It now suffices to apply this observation to a monomorphism $f$ in $\on{mod}_{\on{Ind}(\cat{C})}(A)$ in the diagram
\[\begin{tikzcd}[ampersand replacement=\&,cramped,sep=small]
	{M\otimes I} \& MI \& M \\
	{N \otimes I} \& NI \& N
	\arrow[from=1-1, to=1-2]
	\arrow["{\nabla_{M}}", curve={height=-12pt}, from=1-1, to=1-3]
	\arrow["{f \otimes I}"', from=1-1, to=2-1]
	\arrow[from=1-2, to=1-3]
	\arrow["{\widetilde{f}}"{description}, from=1-2, to=2-2]
	\arrow["f"', from=1-3, to=2-3]
	\arrow[from=2-1, to=2-2]
	\arrow["{\nabla_{N}}"', curve={height=12pt}, from=2-1, to=2-3]
	\arrow[from=2-2, to=2-3]
\end{tikzcd}\]
to find the epimorphism $\widetilde{f}: MI \hookrightarrow NI$. It follows that if $NI=0$ (equivalently, $N$ is in the essential image of $\on{res}(\pi)$) and $M$ is a subobject of $M$, then $MI = 0$. Hence the essential image of $\on{res}(\pi)$ is closed under subobjects, and the claim follows.
\end{proof}

\begin{theorem}[{\cite[Theorem~8.9, Proposition~8.13]{CSZ}}]\label{CSZquotient}
    The radical of $A/J(A)$ vanishes: $J(A/J(A)) = 0$. Hence, $\modw{A/J(A)}$ is an exact $\cat{C}$-module category. The image of (fully faithful, $\cat{C}$-module) restriction functor $\on{res}(\pi)$ along the projection $\pi: A \twoheadrightarrow A/J(A)$ consists of subquotients of objects of the form $Q \lact L$, for $Q \in \projectives{\cat{C}}$ and $L \in \mathbb{L}(A)$.
\end{theorem}

\begin{lemma}\label{radicalcontainment}
    Let $I$ be an ideal in $A$. If $A/I$ is semisimple, then $J(A) \subseteq I$. 
\end{lemma}

\begin{proof}
    Let $\pi: A \twoheadrightarrow A/I$ be the projection algebra morphism. As such, it equips $A/I$ with the structure of an $A$-$A$-bimodule such that $\pi$ becomes an $A$-$A$-bimodule morphism. 
    
    The image $\pi(J(A)) := \on{Im}(J(A)\hookrightarrow A \twoheadrightarrow A/I)$ is the image of such a bimodule morphism, and hence an $A$-$A$-subbimodule of $A/I$. It is easy to verify that it is both left and right annihilated by $I$, and hence a two-sided ideal in $A/I$. Moreover, since $\pi(J(A)^{2})$ is the image of the composite 
    \[
    J(A) \otimes J(A) \xrightarrow{\iota \otimes \iota} A \otimes A \xrightarrow{\mu_{A}} A \xrightarrow{\pi} A/I
    \]
    and $\pi(J(A))^{2}$ is the image of the composite 
    \[
    J(A)\otimes J(A) \xrightarrow{\iota \otimes \iota} A \otimes A \xrightarrow{\pi \otimes \pi} A/I \otimes A/I \xrightarrow{\mu_{A/I}} A/I,
    \]
    we find $\pi(J(A)^{2}) = \pi(J(A))^{2}$, since $\pi$ is an algebra morphism. Inductively, $\pi(J(A)^{m}) = \pi(J(A))^{m}$, showing that $\pi(J(A))$ is nilpotent. Thus $\pi(J(A)) \subseteq J(A/I) = 0$, showing that $\pi(J(A)) = 0$ and hence that $J(A) \subseteq I$.
\end{proof}

\begin{lemma}\label{notorsion}
    Let $A,B \in \Alg{\mathcal{C}}$. Assume that $A$ is a simple algebra. Then for any $A$-$B$-bimodule $M$, the functor $- \otimes_{A} M: \on{mod}_{\mathcal{C}}(A) \rightarrow \on{mod}_{\mathcal{C}}(B)$ reflects zero objects: if $N \in \on{mod}_{\mathcal{C}}(A)$ is such that $N \otimes_{A} M = 0$, then $N = 0$.
\end{lemma}

\begin{proof}
    Since $A$ is a simple algebra, $\on{mod}_{\mathcal{C}}(A)$ is an indecomposable exact $\mathcal{C}$-module category. Hence, for any $M \in \on{mod}_{\mathcal{C}}(A)$, there is a projective object $P \in \mathcal{C}$ such that $P \triangleright M$ is a projective generator for $\on{mod}_{\mathcal{C}}(A)$. If $M \otimes_{A} N = 0$ then $0 = P \triangleright (M \otimes_{A} N) \simeq (P \triangleright M) \otimes_{A} N$. Thus, the right exact functor $- \otimes_{A} N$ sends a projective generator to zero, and hence it sends every object of $\on{mod}_{\mathcal{C}}(A)$ to zero. In particular, $0 = A \otimes_{A} N \simeq N $.
\end{proof}

\subsection{Tensor algebras}

\begin{definition}[{\cite[Definition~3.2]{EKW}}]
        For $A \in \Alg{\cat{C}}$ and an $A$-$A$-bimodule $E$ in $\on{Ind}(\mathcal{C})$, the {\it tensor algebra} $T_{A}(E)$ is the algebra in $\mathbf{Ind}(\mathcal{C})$ given by $T_{A}(E) \coloneqq \bigoplus_{m=0}^{\infty} E^{\otimes_{A} m}$, whose multiplication is given by the maps $E^{\otimes_{A} m} \otimes E^{\otimes_{A} n} \twoheadrightarrow E^{\otimes_{A} m} \otimes_{A} E^{\otimes_{A} n}$.

        The tensor algebra $T_{A}(E)$ is naturally $\mathbb{Z}_{\geq 0}$-graded: for $m \in \mathbb{Z}_{\geq 0}$, via $T_{A}(E)_{m} = E^{\otimes_{A} m}$. 
        We also denote $T_{A}(E)_{\geq r} \coloneqq \bigoplus_{m=r}^\infty T_A(E)_m$
\end{definition}

\begin{proposition}\label{algmapfromtensoralg}
  An algebra morphism $\varphi: T_{A}(E) \rightarrow B$ is uniquely determined by $\varphi_{0}: A \rightarrow B$ and $\varphi_{1}: E \rightarrow B$. The map $\varphi_{0}$ defines an algebra morphism from $A$ to $B$, and $\varphi_{1}$ defines an $A$-$A$-bimodule morphism, where $B$ inherits the $A$-$A$-bimodule structure from $\varphi_{0}$.
  
  Conversely, any pair $(\varphi_{0},\varphi_{1})$, with $\varphi_{0} \in \Alg{\cat{C}}(A,B)$ and $\varphi_{1} \in \bimod(A,A)(E,B)$ defines the algebra morphism $(\varphi_{0},\varphi_{1},\varphi_{1}^{\otimes_{A}2},\cdots): T_{A}(E) \rightarrow B$, which establishes a bijective correspondence.
\end{proposition}

\begin{proposition}
  Since a right $A$-module structure on $M$ is equivalently a morphism $A \rightarrow \hom{M}{M}$ of algebra objects, the correspondence of \autoref{algmapfromtensoralg} yields the correspondence of \cite[Proposition~3.5]{EKW}, between right $T_{A}(E)$-module structures on $M$ and pairs consisting of a right $A$-module structure on $M$ and an $A$-module morphism $M \otimes_{A} E \rightarrow M$.
\end{proposition}

\begin{lemma}\label{compactquotient}
 Let $I$ be an ideal of $T_{A}(E)$ such that $T_{A}(E)/I \in \Alg{\cat{C}}$; in other words, such that $T_{A}(E)/I$ is a compact object of $\on{Ind}(\cat{C})$. Then there is $m \in \mathbb{Z}_{\geq 0}$ such that $\left\langle E^{\otimes_{A}m} \right\rangle := T_{A}(E)_{\geq m}$ is contained in $I$.
\end{lemma}

\begin{proof}
 Given ideals $I,I'$ of $T_{A}(E)$, we denote the supremum of $\setj{I,I'}$ in the poset of two-sided ideals of $T_{A}(E)$ by $I + I'$.
 It exists since the category of $T_{A}(E)$-$T_{A}(E)$-bimodules is abelian - note that this operation makes no reference to the monoidal structure of $\cat{C}$. Similarly, the remainder of this argument does not use the monoidal structure.

 We have an infinite descending chain of inclusions
 \[
 I + \left\langle E^{\otimes_{A}2} \right\rangle \supseteq I + \left\langle E^{\otimes_{A}3} \right\rangle \supseteq \cdots 
 \]
 which yields an infinite chain of inclusions
 \[
 \left(I + \left\langle E^{\otimes_{A}2} \right\rangle\right)/I \supseteq \left(I + \left\langle E^{\otimes_{A}3} \right\rangle\right)/I \supseteq \cdots 
 \]
 of $T_{A}(E)/I$-$T_{A}(E)/I$-bimodules. Since $\on{bimod}_{\cat{C}}(T_{A}(E)/I,T_{A}(E)/I)$ is a finite abelian category, all of its objects are Artinian, and thus the latter chain must terminate, proving the claim.
\end{proof}

\begin{definition}\label{admissibleideal}
 An ideal $I$ of $T_{A}(E)$ is {\it admissible} if $I$ is contained $T_{A}(E)_{\geq 2}$ and there is $m \in \mathbb{Z}_{\geq 0}$ such that $\left\langle E^{\otimes_{A}m} \right\rangle := T_{A}(E)_{\geq m}$ is contained in $I$.
 
 We say that a quotient $T_{A}(E)/I$ is an {\it admissible quotient} if $I$ is admissible.
\end{definition}

\begin{remark}
   By \autoref{admissibleideal}, an admissible quotient of $T_{A}(E)$ is an algebra object in $\cat{C}$, rather than in $\on{Ind}(\cat{C})$. In other words, it is a compact object of $\on{Ind}(\cat{C})$.
\end{remark}

\begin{lemma}
  Let $I$ be an admissible ideal of $T_{A}(E)$ and let $\pi: T_{A}(E) \rightarrow T_{A}(E)/I$ be the associated projection. Let $M \in \Irr(T_{A}(E)/I)$. Then $\on{res}(\pi)(M)E = 0$.
\end{lemma}

\begin{proof}
 By \autoref{ressimple}, we have that $\on{res}(\pi)(M)$ is simple. Thus we either have $\on{res}(\pi)(M)E = 0$ or $\on{res}(\pi)(M)E = M$. Assume the latter. Then
\[\begin{tikzcd}[ampersand replacement=\&,cramped,sep=small]
	{M\otimes (E\otimes_{A} E)} \& {M \otimes_{A} (E\otimes_{A} E)} \& {(M\otimes_{A}E) \otimes_{A} E} \&\& {M\otimes_{A}E} \& M
	\arrow[two heads, from=1-1, to=1-2]
	\arrow["\simeq", from=1-2, to=1-3]
	\arrow["{\nabla^{1}_{M}\otimes_{A}E}", from=1-3, to=1-5]
	\arrow["\simeq"', from=1-3, to=1-5]
	\arrow["{\nabla^{1}_{M}}", from=1-5, to=1-6]
	\arrow["\simeq"', from=1-5, to=1-6]
\end{tikzcd}\]
shows that $MT_{A}(E)_{\geq 2} = M$, and inductively one shows that $MT_{A}(E)_{\geq k} = M$ for all $k$. However, this contradicts $MI = 0$, by \autoref{compactquotient}. Hence we must have $\on{res}(\pi)(M)E = 0$.
\end{proof}

\begin{corollary}\label{simpleskillarrows}
 Let $I$ be an admissible ideal of $T_{A}(E)$. Let 
 \[
 \phi: (T_{A}(E)/I)/(\left\langle E\right\rangle/I) \xiso A
 \]
 be the obvious induced isomorphism of algebra objects. Then the fully faithful functor
 \[
 \mathbb{L}(\modca) \rightarrow \mathbb{L}(\modw{T_{A}(E)/I})
 \]
 restricted from $\on{res}(\phi)$ to the subcategory of semisimple objects is an equivalence of categories.
\end{corollary}

Using \autoref{simpleskillarrows}, one can deduce the claim below; we opt to provide a separate proof. 

\begin{proposition}\label{radicalofadmissiblequotient}
   Let $B$ be a semisimple algebra in $\cat{C}$, let $E \in \bimod(B,B)$ and let $I$ be an admissible ideal of $T_{B}(E)$. Then $J(T_{B}(E)/I) = \left\langle E\right\rangle/I$. 
\end{proposition}

\begin{proof}
 As a consequence of \autoref{notorsion}, for any $n \geq 0$ we have an equality of ideals $\left\langle E\right\rangle^{n} = \left\langle E^{\otimes_B n}\right\rangle$ of ideals in $T_{B}(E)$.
 By \autoref{compactquotient}, we may choose $m$ such that $\left\langle E \right\rangle^{m} \subseteq I$. From the commutativity of 
\[\begin{tikzcd}[column sep=scriptsize]
	{I^{\otimes n}} & {T_{B}(E)^{\otimes n}} & {T_{B}(E)} \\
	{\left\langle E\right\rangle^{\otimes n}} & {T_{B}(E)^{\otimes n}} & {T_{B}(E)} \\
	{(\left\langle E\right\rangle/I)^{\otimes n}} & {(T_{B}(E)/I)^{\otimes n}} & {T_{B}(E)/I}
	\arrow[hook, from=1-1, to=1-2]
	\arrow[hook, from=1-1, to=2-1]
	\arrow["{\mu_{n}}", from=1-2, to=1-3]
	\arrow[from=1-2, to=2-2]
	\arrow[from=1-3, to=2-3]
	\arrow[hook, from=2-1, to=2-2]
	\arrow[two heads, from=2-1, to=3-1]
	\arrow["{\mu_{n}}", from=2-2, to=2-3]
	\arrow[two heads, from=2-2, to=3-2]
	\arrow[two heads, from=2-3, to=3-3]
	\arrow[hook, from=3-1, to=3-2]
	\arrow["{\mu_{n}}", from=3-2, to=3-3]
\end{tikzcd}\]
we obtain the familiar equation $(\left\langle E \right\rangle/I)^{n} = (\left\langle E\right\rangle^{n}+I)/I$, for any $n$. This shows $(\left\langle E \right\rangle/I)^{m} = 0$ for the previously chosen $m$, hence $\left\langle E \right\rangle/I$ is nilpotent by \autoref{CSZEO}. Thus $\left\langle E \right\rangle/I\subseteq J(T_{B}(E))$ On the other hand there is an obvious isomorphism $(T_{B}(E)/I)/(\left\langle E\right\rangle/I) \simeq B$, thus the quotient $(T_{B}(E)/I)/(\left\langle E\right\rangle/I)$ is semisimple. By \autoref{radicalcontainment} this shows $J(T_{B}(E)) \subseteq \left\langle E \right\rangle/I$, proving the result.
\end{proof}

\subsection{Division algebras}

 \begin{definition}
 An algebra $A$ in a finite tensor category $\mathcal{C}$ is a {\it division algebra} if $A$ is simple as a right $A$-module.
 \end{definition}

 For a survey on properties of division algebras in monoidal categories, see \cite{KW}. We will consider a slight generalisation of \cite[Proposition~3.6]{KW}:

\begin{proposition}\label{prop:divisionalgs}
 Let $A$ be an algebra in a finite tensor category. The following are equivalent:
 \begin{enumerate}
  \item\label{division} $A$ is a division algebra.
  \item\label{endexactsimple} $A \simeq \underline{\on{End}}_{\mathcal{M}}(L)$ for a simple object $L$ in an exact $\mathcal{C}$-module category $\mathcal{M}$.
  \item\label{endsimple} $A \simeq \underline{\on{End}}_{\mathcal{M}}(L)$ for a simple object $L$ in a $\mathcal{C}$-module category $\mathcal{M}$.
 \end{enumerate}
\end{proposition}

\begin{proof}
  Assume that $A$ is a division algebra. We have $A \simeq \underline{\on{End}}_{A}(A)$, proving that \autoref{division} implies \autoref{endsimple}. Since a division algebra is simple (and hence semisimple), we find that $\on{mod}_{\mathcal{C}}(A)$ is an exact $\mathcal{C}$-module category by \autoref{CSZBijection}, so \autoref{division} also implies \autoref{endexactsimple}.

  Let $\mathcal{M}$ be an exact $\mathcal{C}$-module category. Let $L$ be a simple object of $\mathcal{M}$ and let $\mathcal{M}_{L}$ be the indecomposable summand of $\mathcal{M}$ containing $L$. Then we find an equivalence of $\mathcal{C}$-module categories
  \[
   \mathcal{M}_{L} \xrightarrow[\sim]{\underline{\on{Hom}}_{\mathcal{M}}(L,-)} \on{mod}_{\mathcal{C}} \underline{\on{End}}_{\mathcal{M}}(L)
  \]
   which maps $L$ to $\underline{\on{End}}_{\mathcal{M}}(L)$. Since $L$ is a simple object of $\mathcal{M}_{L}$, it follows that $\underline{\on{End}}_{\mathcal{M}}(L)$ is a simple $\underline{\on{End}}_{\mathcal{M}}(L)$-module. Thus, \autoref{endexactsimple} implies \autoref{division}.

   Now let $\mathcal{M}$ be a $\mathcal{C}$-module category which is not necessarily exact, and let $L$ be a simple object of $\mathcal{M}$. Similarly to above, we may assume $\mathcal{M}$ to be indecomposable. Let $A \in \Alg{\cat{C}}$ be such that $\cat{M} \simeq \modca$. Since the category $\modw{A/J(A)}$
   is a full $\cat{C}$-module subcategory of $\mathcal{M}$ (see \autoref{CSZquotient}), we find an isomorphism $\underline{\on{End}}_{\mathcal{M}}(L) \simeq \underline{\on{End}}_{A/J(A)}(L)$, by \autoref{enrichedhom}, which, by \autoref{EnrichmentsKock}, is an isomorphism of algebra objects.

   Since we have already established that $\underline{\on{End}}_{A/J(A)}(L)$ is a division algebra, so is $\underline{\on{End}}_{\mathcal{M}}(L)$, establishing the result.
\end{proof}

Note that the above proposition can be used to compute all division algebras of $\cat{C}$ up to isomorphism -- not just Morita equivalence -- once we are handed a classification of its exact $\cat{C}$-module categories.

\begin{example}\label{eg:vecdivisionalg}
    Let $\cat{C}=\vect$ be the category of finite-dimensional vector spaces.
    It is immediate that the only indecomposable semisimple module category over $\cat{C}$ is $\cat{C}$ over itself.
    We thus obtain the well-known fact that the only division algebra in $\cat{C}$ is the trivial algebra (recall that we always work over an algebraically closed field).
\end{example}

\begin{example}\label{eg:Fibdivisionalg}
    Let $\cat{C} = \Fib$ be the Fibonacci fusion category, defined by $\on{Irr}(\cat{C}) = \{ \mathbb{1}, \Phi \}$ and the fusion rule
    \[
    \Phi \otimes \Phi \simeq \mathbb{1} \oplus \Phi.
    \]
    Technically, there are two fusion categories with this fusion rule, differentiated through $\Phi$ having different quantum dimension q$\dim(\Phi) = \frac{1\pm\sqrt{5}}{2}$; both can be realised by fusion subcategories associated to $U_{q}(\mathfrak{sl}_{2})$ for $q$ at a primitive 10th-root of unity. 
    One may pick either of these two for the upcoming discussion, though $\Fib$ is usually used to denote the one with q$\dim(\Phi) = \frac{1+\sqrt{5}}{2}$.

    It is known that $\cat{C}$ has exactly one indecomposable exact module category, namely $\cat{C}$ acting on itself (one way to show this is to prove in a similar fashion to \cite{EK} that there is no other irreducible $\mathbb{Z}_+$-module over the fusion ring of $\cat{C}$ other than the fusion ring itself, and then apply \autoref{lem:rigidityofregularmodule}).
    
    We conclude that there are exactly two non-isomorphic (but Morita equivalent) division algebras in $\cat{C}$: the trivial algebra $\mathbb{1}$ and the algebra $\underline{\on{End}}_{\cat{C}}(\Phi) \simeq  \Phi \otimes \Phi$ (with algebra structure given exactly in \cite[Example 7.8.4]{EGNO}).
\end{example}

\begin{example}\label{eg:Z2divisionalg}
    Let $\vect_G$ be the category of finite-dimensional $G$-graded vector spaces.
    Its indecomposable semisimple module categories are classified by the data of a subgroup $L \subseteq G$ and a class of 2-cocycle $\psi \in H^2(L,\mathbb{k}^*)$ \cite[Example 7.4.10]{EGNO}.
    Assume $\on{char}(\mathbb{k})=0$. 
    For $G = \mathbb{Z}/2\mathbb{Z} = \{0,1\}$, we get three indecomposable semisimple module categories:
    \begin{itemize}
        \item When $L= \{0\}$, there is only one class of 2-cocycle $\psi = 1$. This corresponds to the regular module category $\vect_G$.
        \item When $L = G$, we get two choices $\psi = \pm 1$. These correspond to $\vect$ over $\vect_G$, one induced by the forgetful functor $\vect_G \to \vect$ and the other has action twisted by $\psi = -1$.
    \end{itemize}
    From this we conclude that there are exactly three division algebras up to isomorphism:
    \begin{itemize}
        \item While $\cat{M} = \vect_G$ has two simple objects $\mathbb{k}_0$ and $\mathbb{k}_1$, we only obtain one division algebra (cf.\ \autoref{eg:Fibdivisionalg}), since:
        \[
        \underline{\on{End}}_{\cat{M}}(\mathbb{k}_0) \simeq \underline{\on{End}}_{\cat{M}}(\mathbb{k}_1) \simeq \mathbb{k}_0.
        \]
        \item On the other hand, we get two non-isomorphic division algebras for the different choices of $\psi$: the standard graded group algebra $\mathbb{k}[G] \coloneqq \mathbb{k}[x]/\langle x^2-1 \rangle$ and the twisted graded group algebra $\mathbb{k}^\psi[G]  \coloneqq \mathbb{k}[x]/\langle x^2+1\rangle$ (in both cases, $x$ is in degree 1, and $1$ is in degree 0).
    \end{itemize}
\end{example}

\subsection{Artin--Wedderburn theorem}
In this section we show that any semisimple algebra in a finite tensor category is Morita equivalent to a finite direct product of division algebras.

\begin{definition} \label{defn:C-orbitsimples}
 Let $\cat{M}$ be an exact $\cat{C}$-module category. Consider the following equivalence relations on the set $\on{Irr}(\cat{M})$ of representatives (of isomorphism classes) of simple objects in $\cat{M}$:
 \begin{itemize}
     \item $L \sim_{H} L'$ if and only if there is $V \in \cat{C}$ such that $L$ is a subquotient of $V \lact L'$.
     \item $L \sim_{F} L'$ if and only if there is $P \in \projectives{\cat{C}}$ such that $\on{Hom}_{\cat{M}}(P\lact L, L') \neq 0$. 
 \end{itemize} 
 The following proposition shows that these two equivalence relations are the same, and henceforth will simply be denoted by $\sim$.
\end{definition}

\begin{proposition}
    The two equivalence relations introduced above coincide: $\sim_{F} = \sim_{H}$.
\end{proposition}

\begin{proof}
    Clearly, if $L \sim_{F} L'$ then $L'$ is a quotient, and hence in particular subquotient, of $P\lact L$, for some $P \in \cat{C}$, showing $L \sim_{H} L'$.

    If $L \sim_{H} L'$, let $V$ be such that the composition multiplicity $[V\lact L':L]$ is non-zero. Let $Q_{\mathbb{1}}\in \projectives{\cat{C}}$ be such that there is an epimorphism $q: Q_{\mathbb{1}}\twoheadrightarrow \mathbb{1}$. 

    Then $Q_{\mathbb{1}}\lact L$ is a projective object with an epimorphism $q \lact L$ onto $L$, so that $L$ is a direct summand of $\on{top}(Q_{\mathbb{1}}\lact L)$. Using the equality $[M:\on{top}P] = \on{dim}\on{Hom}(P,M)$ and $[V\lact L':L] \neq 0$ we find $\on{dim}\on{Hom}_{\cat{M}}(Q_{\mathbb{1}}\lact L,V\lact L') \neq 0$ and hence $\on{dim}\on{Hom}_{\cat{M}}(\leftdual{V}\otimes Q_{\mathbb{1}}\lact L,L') \neq 0$, showing that $L \sim_{F} L'$.
\end{proof}

\begin{proposition}\label{endproduct}
    Let $L,L'$ be such that $L \not\sim L'$. Then $\hom{L}{L'} = 0$. Hence, we find an isomorphism $\underline{\on{End}}_{\cat{M}}(L\oplus L') \simeq \underline{\on{End}}_{\cat{M}}(L) \times \underline{\on{End}}_{\cat{M}}(L')$ in $\Alg{\cat{C}}$.
\end{proposition}

\begin{proof}
    For $P \in \projectives{\cat{C}}$ we find $\on{Hom}_{\cat{C}}(P,\hom{L}{L'}) = \on{Hom}_{\cat{M}}(P\lact L,L') = 0$. 
\end{proof}

\begin{proposition}\label{basicsemisimple}
   Let $\cat{M}$ be an exact $\cat{C}$-module category. Choose representatives $L_{1},\ldots,L_{r}$ for the equivalence classes of $\on{Irr}(\cat{M})/\sim$. Then $\bigoplus_{i=1}^{r}L_{i}$ is a $\cat{C}$-projective generator for $\cat{M}$.
\end{proposition}

\begin{proof}
    Since $\cat{M}$ is exact, all of its objects are $\cat{C}$-projective. Hence it suffices to show that the functor $\hom{\bigoplus_{i=1}^{r}L_{i}}{-}$ reflects zero objects. Let $P$ be a projective generator for $\cat{C}$. Then for any $L \in \on{Irr}(\cat{M})$ we have $\on{Hom}_{\cat{M}}(P\lact \bigoplus_{i=1}^{r} L_{i},L) \neq 0$, by definition of $\bigoplus_{i=1}^{r}L_{i}$. Thus, $P\lact \bigoplus_{i=1}^{r} L_{i}$ is a projective generator for $\cat{M}$. Hence
    \[
    \on{Hom}_{\cat{C}}(P,\hom{\bigoplus_{i=1}^{r}L_{i}}{-}) \simeq \on{Hom}_{\cat{C}}(P\lact \bigoplus_{i=1}^{r}L_{i}, -)
    \]
    reflects zero objects, and thus so does $\hom{\bigoplus_{i=1}^{r}L_{i}}{-}$.
\end{proof}

\begin{proposition}\label{ArtinWedderburn}
   Any semisimple algebra is Morita equivalent to a finite direct product of division algebras.  
\end{proposition}

\begin{proof}
   Let $B$ be a semisimple algebra and let $\cat{M} = \modw{B}$.
    Choose representatives $L_{1},\ldots,L_{r}$ for the equivalence classes of $\on{Irr}(\cat{M})/\sim$. Then by \autoref{basicsemisimple}, we have $\cat{M} \simeq \modw{\underline{\on{End}}_{\cat{M}}( \bigoplus_{i=1}^{r}L_{i})}$, showing that $B$ is Morita equivalent to ${\underline{\on{End}}_{\cat{M}}( \bigoplus_{i=1}^{r}L_{i})}$. 
    
    By \autoref{endproduct}, we find ${\underline{\on{End}}_{\cat{M}}( \bigoplus_{i=1}^{r}L_{i})} \simeq \prod_{i=1}^{r} {\underline{\on{End}}_{\cat{M}}(L_{i})}$, and ${\underline{\on{End}}_{\cat{M}}(L_{i})}$ is a division algebra for every $i$, by \autoref{prop:divisionalgs}.
\end{proof}

\section{Species, quivers and path algebras}\label{species}

\begin{definition}\label{defn:speciesandquiver}
A {\it quasi-species} in $\cat{C}$, or $\cat{C}$-quasi-species, consists of a triple $(I, (B_i)_{i \in I} (E_{ij})_{i,j \in I})$, where:
\begin{itemize}
  \item  $I$ is a finite set;
  \item  For each $i \in I$, $B_{i}$ is a simple algebra in $\cat{C}$;
  \item  For each $i,j \in I$, $E_{ij}$ is a (possibly zero) $B_{i}$-$B_{j}$-bimodule in $\cat{C}$.
\end{itemize}
 A {\it $\cat{C}$-species} is a $\cat{C}$-quasi-species such that $B_{i}$ is a division algebra, for all $i \in I$. A {\it $\cat{C}$-quiver} is a $\cat{C}$-species such that $B_{i} = \mathbb{1}$, for all $i \in I$.

 Given a species $(I,(B_{i})_{i \in I}, (E_{ij})_{i,j\in I})$, we denote by $B$ the product algebra $\prod_{i \in I} B_{i}$, and we denote by $E$ the $B$-$B$-bimodule $\bigoplus_{i,j \in I} E_{ij}$, with bimodule structure obtained via \autoref{modproduct}. We refer to the pair $(B,E)$ as the {\it type} of $(I,(B_{i})_{i \in I}, (E_{ij})_{i,j\in I})$.
\end{definition}

\begin{remark}
    What we call a $\cat{C}$-quasi-species is simply called a $\cat{C}$-species in \cite{QZ_fusionstable} (in fact, they even allow $B_i$ to be semisimple). Given \autoref{ArtinWedderburn}, we would prefer to call these $\cat{C}$-quasi-species, so that the notion of $\cat{C}$-species fits well with the classical case (e.g.\ quivers are typically given by $B_i = \mathbb{k}$; not matrix algebras).
\end{remark}

\begin{definition}\label{pathalgebra}
  Given a quasi-species $(I,(B_{i})_{i \in I}, (E_{ij})_{i,j\in I})$ with type $(B,E)$, we define its {\it path algebra} as the tensor algebra $T_{B}(E)$ in $\on{Ind}(\cat{C})$.
\end{definition}

The following motivates the terminology of \autoref{pathalgebra}:

\begin{proposition}\label{gradedtensor}
   Let $B = \prod_{i=1}^{n} D_{i}$ be a finite product of simple algebras in $\cat{C}$ and $E \in \bimod(B,B)$. 
   Using \autoref{modproduct} we can decompose $E$ as $\bigoplus_{i,j} E_{ij}$ where $E_{ij} \in \bimod(D_{i},D_{j})$ the B-B-bimodule.

   For the standard grading of the tensor algebra $T_{B}(E) = \bigoplus_{i=0}^{\infty} T_{B}(E)_{k}$, where $T_{B}(E)_{k} = E^{\otimes_{B}k}$, we find $T_{B}(E)_{k} = \bigoplus_{(i_{1},\ldots,i_{k})\in\setj{1,\ldots,n}^{k}} E_{(i_{1},\ldots,i_{m})}$, where we set
   $E_{(i_{1},\ldots,i_{m})} := E_{i_{n-1}i_{n}} \otimes_{D_{i_{n-1}}} E_{i_{n-2}i_{n-1}} \otimes_{D_{i_{n-2}}} \cdots \otimes_{D_{i_{2}}} E_{i_{1}i_{2}}$.
\end{proposition}

\begin{remark}
   Observe that for a quiver in $\cat{C} = \vect$, the space $E_{(i_{1},\ldots,i_{m})}$ consists of linear combinations of paths traversing the sequence $(i_{1},\ldots,i_{m})$ of vertices in the quiver.
\end{remark}

The following proposition tells us up to Morita equivalence, we will only need to concern ourselves with $\cat{C}$-species.
\begin{proposition}\label{quasirepresentative}
  Let $\big(I,(B_{i})_{i \in I}, (E_{ij})_{i,j\in I}\big)$ be a quasi-species in $\cat{C}$. There is a species $\big(I,(B'_{i})_{i \in I},(E'_{ij})_{i,j \in I}\big)$ such that the path algebras of $(I,(B_{i})_{i \in I}, (E_{ij})_{i,j\in I})$ and of $\big(I,(B'_{i})_{i \in I},(E'_{ij})_{i,j \in I}\big)$ are Morita equivalent.
\end{proposition}

\begin{proof}
  Using \autoref{ArtinWedderburn}, we may let $B'_{i}$ be a division algebra Morita equivalent to $B_{i}$. Hence, by \autoref{modproduct}, $B'$ is Morita equivalent to $B$. Let $M \in \bimod(B,B')$ and $N \in \bimod(B',B)$ be the bimodules realising the Morita equivalence. Let $E'_{ij} = N\otimes_{B} E_{ij} \otimes_{B} M$. Then by \cite[Proposition~3.5]{EKW}, the path algebras $T_{B}(E)$ and $T_{B'}(E')$ are Morita equivalent. But since $B_{i}'$ is a division algebra, the quasi-species $\big(I,(B'_{i})_{i \in I},(E'_{ij})_{i,j \in I}\big)$ is a species.
\end{proof}

\begin{definition}
    Given a $\cat{C}$-species $\big(I,(B_{i})_{i \in I}, (E_{ij})_{i,j\in I}\big)$, we define its {\it underlying directed graph} as the directed graph whose vertex set is $I$, where, for $i,j\in I$, we have an arrow $i \to j$ if and only if $E_{ij} \neq 0$.

    We say that a $\cat{C}$-species is {\it acyclic} if and only if its underlying directed graph is acyclic (i.e.\ there is no non-trivial path that starts and end at the same vertex).
\end{definition}

Borrowing the convention from finite-dimensional algebras, we shall simply present a $\cat{C}$-species $(I,(B_{i})_{i \in I}, (E_{ij})_{i,j\in I}\big)$ by its underlying directed graph together with the vertex and arrow labels.
Below are some examples.

\begin{example}
    Let $\cat{C}=\vect$ be the category of finite-dimensional vector spaces.
    Recall that the only division algebra in $\cat{C}$ is the trivial algebra (we always work over an algebraically closed field).
    As such, all $\cat{C}$-species are $\cat{C}$-quivers.
    
    Consider the $\cat{C}$-quiver defined by $I=\{1,2\}$, $B_1 = \mathbb{k} = B_2$, $E_{12}=\mathbb{k}^2$ and $E_{21}=0$.
    Its underlying directed graph is simply a directed graph with two vertices $1$ and $2$, and a single arrow from $1$ to $2$.
    However, the avid reader may realise that the path algebra of this $\cat{C}$-quiver is (isomorphic to) the path algebra of the Kronecker quiver, which is instead a quiver that has two arrows going from $1$ to $2$.
    Indeed, to encode the full data of the $\cat{C}$-quiver, both the vertices and the arrow have to be labelled as follows:
    \[
    \begin{tikzcd}
        \mathbb{k} \ar[r, "\mathbb{k}^2"] & \mathbb{k}
    \end{tikzcd}.
    \]
    In the setting of finite-dimensional vector spaces, there is no loss of data in using $n$ parallel arrows versus using a single arrow labelled by an $n$-dimensional vector space, as a vector space is completely determined by its dimension up to isomorphism.
    As we shall see, for general finite tensor categories (even fusion categories) it is not always possible to replace a labelled arrow by multiple arrows.
\end{example}

\begin{example} \label{eg:pentagonCoxquiver}
    Recall from \autoref{eg:Fibdivisionalg} the fusion category $\Fib$ and its division algebras.
    The following are, respectively, a $\cat{C}$-quiver and a $\cat{C}$-species (whose path algebras are Morita equivalent):
    \[
    \begin{tikzcd}
        \mathbb{1} \ar[r, "\Phi"] & \mathbb{1}
    \end{tikzcd};
    \quad
        \begin{tikzcd}
        \underline{\on{End}}_{\cat{C}}(\Phi) \ar[r, "\Phi\otimes\Phi"] & \underline{\on{End}}_{\cat{C}}(\Phi)
    \end{tikzcd}.
    \]
    Comparing with our previous example, the $\cat{C}$-quiver on the left has $\Phi$ labelling its single arrow, which would've required a ``golden-ratio number of arrows'', as $\FPdim(\Phi)=\text{golden ratio}$.
\end{example}

\begin{example} \label{eg:Z2species}
    Let $G \coloneqq \mathbb{Z}/2\mathbb{Z}$ and let $\cat{C} \coloneqq \vect_{G}$; cf.\ \autoref{eg:Z2divisionalg}.
    The following are some examples of $\cat{C}$-species, where the bimodule structures on the arrow labels are the ``obvious'' ones (no pair of which have Morita equivalent path algebras).
    For notation simplicity we will denote $D \coloneqq \mathbb{k}[G]$ and $D^\psi \coloneqq \mathbb{k}^\psi[G]$ the division algebras from \autoref{eg:Z2divisionalg}.
    \begin{enumerate}
        \item \label{item:twoD4copies}
        \begin{tikzcd}
            \mathbb{k}_0 \ar[r, "\mathbb{k}_0"] & \mathbb{k}_0 \ar[r, "\mathbb{k}_0"] & \mathbb{k}_0 \\
            {} & \mathbb{k_0} \ar[u, "\mathbb{k}_0"]
        \end{tikzcd} \qquad (a $\cat{C}$-quiver);
        \item \label{item:stdZ2action}
        \begin{tikzcd}
            \mathbb{k}_0 \ar[r, "D"] & D \ar[r, "D"] & D
        \end{tikzcd} \qquad (not a $\cat{C}$-quiver);
        \item \label{item:twistedZ2action}
        \begin{tikzcd}
            \mathbb{k}_0 \ar[r, "D^\psi"] & D^\psi \ar[r, "D^\psi"] & D^\psi 
        \end{tikzcd} \qquad (not a $\cat{C}$-quiver).
    \end{enumerate}
\end{example}

\begin{proposition}
   Let $\big(I,(B_{i})_{i \in I}, (E_{ij})_{i,j\in I}\big)$ be a quasi-species in $\cat{C}$. Its path algebra is a compact object of $\mathbf{Ind}(\mathcal{C})$ -- equivalently, an object of $\mathcal{C}$ -- if and only if $\big(I,(B_{i})_{i \in I}, (E_{ij})_{i,j\in I}\big)$ is acyclic.
\end{proposition}

\begin{proof}
   As in \autoref{gradedtensor}, we have $T_{B}(E) \simeq \bigoplus_{m=0}^{\infty} T_{B}(E)_{m}$, where $T_{B}(E)_{m} \simeq \bigoplus_{(i_{1},\ldots,i_{m}) \in I^{m}} E_{(i_{1},\ldots,i_{m})}$ and we set
   $E_{(i_{1},\ldots,i_{m})} := E_{i_{n-1}i_{n}} \otimes_{B_{i_{n-1}}} E_{i_{n-2}i_{n-1}} \otimes_{B_{i_{n-2}}} \cdots \otimes_{B_{i_{2}}} E_{i_{1}i_{2}}$.
   
   By \autoref{notorsion}, for $(i_{1},\ldots,i_{m}) \in I^{m}$, we have $E_{(i_{1},\ldots,i_{m})} = 0$ if and only if there is $k \in \setj{1,\ldots,m}$ such that $E_{i_{k-1}i_{k}} = 0$. For $n > |I|$, every sequence in $I^{n}$ contains a repetition, and hence $T_{B}(E)_{n} = 0$ for all $n > |I|$ if and only if $\big(I,(B_{i})_{i \in I}, (E_{ij})_{i,j\in I}\big)$
   is acyclic. On the other hand, the former condition is equivalent to compactness of $T_{B}(E)$.
\end{proof}

We now fix an algebra $A$ and denote the radical $J(A)$ of $A$ (in the sense of \autoref{internalradical}) simply by $J$. The following condition is analogous to \cite[Proposition~4.1.10]{Ben}, \cite{Ber}:
\begin{definition}\label{def:radicalsplit}
  We say that $A$ is {\it radical-split} if the following two conditions hold:   
  \begin{enumerate}
      \item The projection $\pi_{0}: A \rightarrow A/J$ of algebras admits a section $\sigma_{0}$ in $\Alg{\cat{C}}$.
      \item The projection $\pi_{1}: J \rightarrow J/J^{2}$ of $A/J$-$A/J$-bimodules admits a section $\sigma_{1}$ in $\bimod(A/J,A/J)$.
  \end{enumerate}

  We say that the pair $(\sigma_{0},\sigma_{1})$ constitutes a {\it radical-splitting} for $A$.

  By \autoref{algmapfromtensoralg}, a radical-splitting for $A$ gives rise to a morphism of algebras from $T_{A/J}(J/J^{2})$ to $A$. We denote this morphism by $\Phi(\sigma_{0},\sigma_{1})$.
\end{definition}

\begin{remark}\label{radicalsplittingtensor}
    Given a semisimple algebra $B$, together with $E \in \bimod(B,B)$ and an admissible ideal $I$ of $T_{B}(E)$, the algebra $T_{B}(E)/I$ is radical-split. 
    
    The splitting $(T_{B}(E)/I)/J(T_{B}(E)/I) \rightarrow T_{B}(E)/I$ is obtained as 
    \[
    (T_{B}(E)/I)/J(T_{B}(E)/I) = (T_{B}(E)/I)/(\left\langle E\right\rangle/I) \simeq B \hookrightarrow T_{B}(E)\twoheadrightarrow T_{B}(E)/I.
    \]
    The first equality above follows from \autoref{radicalofadmissiblequotient}. 

    The splitting $J(T_{B}(E)/I)/J^{2}(T_{B}(E)/I)\rightarrow J(T_{B}(E)/I)$, is given by 
    \[
    (\left\langle E\right\rangle/I)/(\langle E^{\otimes 2}\rangle/I) \simeq \left\langle E\right\rangle / \langle E^{\otimes 2}\rangle \hookrightarrow \left\langle E\right\rangle \twoheadrightarrow \left\langle E\right\rangle /I. 
    \]
    
    Here we use the obvious isomorphism
    $\left\langle E\right\rangle / \langle E^{\otimes 2}\rangle \simeq T_{B}(E)_{1} \simeq E$ to obtain an induced embedding of $B$-$B$-bimodules
    $\left\langle E\right\rangle / \langle E^{\otimes 2}\rangle \hookrightarrow \left\langle E\right\rangle$ from the embedding $E \hookrightarrow \left\langle E\right\rangle$.
\end{remark}

\begin{lemma}\label{kertensor}
 Let $\mathcal{C}$ be a ring category and $f: U \rightarrow V$ and $g: W \rightarrow Y$ morphisms of $\mathcal{C}$, we have
 \[
 \on{Ker}(f \otimes g) = \on{Ker}(f) \otimes W + U \otimes \on{Ker}(g).
 \]
\end{lemma}

\begin{lemma}\label{sumcover}
    Consider a diagram of the form
\[\begin{tikzcd}
	& {M'} & {M''} \\
	0 & {N'} & N & {N''} & 0
	\arrow["{p'}", two heads, from=1-2, to=2-2]
	\arrow["{p''}", from=1-3, to=2-3]
	\arrow[from=2-1, to=2-2]
	\arrow["f", hook, from=2-2, to=2-3]
	\arrow["g", two heads, from=2-3, to=2-4]
	\arrow[from=2-4, to=2-5]
\end{tikzcd}\]
in an abelian category, with exact rows. If $g\circ p''$ is surjective, then $M' \oplus M'' \xrightarrow{\left(\begin{smallmatrix} f \circ p' & p'' \end{smallmatrix} \right)} N$ is an epimorphism.
\end{lemma}

\begin{theorem}[weak Gabriel's theorem for finite tensor categories]\label{weakGabriel}
 Let $A$ be an algebra in a finite tensor category $\cat{C}$ and let $J$ denote its radical. If $(\sigma_{0},\sigma_{1})$ is a radical-splitting for $A$, then the morphism $\Phi(\sigma_{0},\sigma_{1}): T_{A/J}(J/J^{2}) \rightarrow A$ is a surjective morphism of algebras.
 Moreover, $\on{Ker}(\Phi(\sigma_{0},\sigma_{1}))$ is an admissible ideal of $T_{A/J}(J/J^{2})$, in the sense of \autoref{admissibleideal}.

 Hence, any radical-split algebra in $\cat{C}$ is an admissible quotient of the path algebra of a quasi-species in $\cat{C}$ of type $(A/J, J/J^{2})$.
\end{theorem}

\begin{proof}
 Applying inductively \autoref{kertensor}, we obtain for every integer $n>0$ the exact sequence
\[\begin{tikzcd}
	0 & { \sum_{k=1}^{n} J^{\otimes_{A/J}\; k-1} \otimes_{A/J} J^{2} \otimes_{A/J} J^{\otimes_{A/J} \; n-k}} & {J^{\otimes_{A/J}\; n}} & {(J/J^{2})^{\otimes_{A/J}\; n}} & 0
	\arrow[from=1-1, to=1-2]
	\arrow[from=1-2, to=1-3]
	\arrow["{\pi_{1}^{\otimes_{A/J} \; n}}", from=1-3, to=1-4]
	\arrow[from=1-4, to=1-5]
\end{tikzcd}\]
In the diagram below, the epimorphisms represented by solid vertical arrows are obtained using $n$-fold multiplication, the vertical dashed arrow is the induced epimorphism on cokernels, and $\gamma$ is the split epimorphism obtained by applying the splitting lemma to the splitting of the exact sequence in the top row given by $\sigma_{1}^{\otimes_{A/J} n}$.
\[\begin{tikzcd}
	0 & { \sum_{k=1}^{n} J^{\otimes_{A/J}\; k-1} \otimes_{A/J} J^{2} \otimes_{A/J} J^{\otimes_{A/J} \; n-k}} & {J^{\otimes_{A/J}\; n}} & {(J/J^{2})^{\otimes_{A/J}\; n}} & 0 \\
	0 & {J^{n+1}} & {J^{n}} & {J^{n}/J^{n+1}} & 0
	\arrow[from=1-1, to=1-2]
	\arrow["\kappa", shift left, from=1-2, to=1-3]
	\arrow[two heads, from=1-2, to=2-2]
	\arrow["\gamma", shift left=2, dashed, from=1-3, to=1-2]
	\arrow["{\pi^{\otimes_{A/J} \; n}}", shift left, from=1-3, to=1-4]
	\arrow[two heads, from=1-3, to=2-3]
	\arrow["{\sigma_{1}^{\otimes_{A/J} \; n}}", shift left=2, dashed, from=1-4, to=1-3]
	\arrow[from=1-4, to=1-5]
	\arrow[dashed, two heads, from=1-4, to=2-4]
	\arrow[from=2-1, to=2-2]
	\arrow[from=2-2, to=2-3]
	\arrow[from=2-3, to=2-4]
	\arrow[from=2-4, to=2-5]
\end{tikzcd}\]

By the commutativity of the left solid square, precomposing the map $\psi: J^{\otimes n} \twoheadrightarrow J^{n}/J^{n+1}$ obtained by composing solid arrows in the diagram above, with the idempotent $\kappa \circ \gamma$ is zero. Hence, since $\on{id}_{J^{\otimes n}} = \kappa \circ \gamma + \sigma_{1}^{\otimes_{A/J} n} \circ \pi_{1}^{\otimes_{A/J} n}$, we find that $\psi$ coequalizes $\on{id}_{J^{\otimes_{A/J} n}}$ and $\sigma_{1}^{\otimes_{A/J} n} \circ \pi_{1}^{\otimes_{A/J} n}$, so $\psi \circ \sigma_{1}^{\otimes_{A/J} n}$ coincides with the dashed vertical arrow, and hence it is an epimorphism.

By definition of $\Phi(\sigma_{0},\sigma_{1})$, the following diagram commutes:
\[\begin{tikzcd}
	{(J/J^{2})^{\otimes_{A/J}\; n}} & {J^{\otimes_{A/J}\; n}} & {J^{n}} \\
	{T_{A/J}(J/J^{2})_{n}} & {T_{A/J}(J/J^{2})} & A
	\arrow["{\sigma_{1}^{\otimes_{A/J} \; n}}", from=1-1, to=1-2]
	\arrow["{=}"', from=1-1, to=2-1]
	\arrow[two heads, from=1-2, to=1-3]
	\arrow[hook, from=1-3, to=2-3]
	\arrow[hook, from=2-1, to=2-2]
	\arrow["{\Phi(\sigma_{0},\sigma_{1})}"', from=2-2, to=2-3]
\end{tikzcd}\]
By \cite[Lemma~8.8]{CSZ}, there exists a maximal $m$ such that $J^{m} \neq 0$. Combining this with an inductive application of \autoref{sumcover}, we find that the surjectivity of the maps $T_{A/J}(J/J^{2})_{n} \twoheadrightarrow J^{n}/J^{n+1}$ entails the surjectivity of $\Phi$.

From commutativity of
\[\begin{tikzcd}[ampersand replacement = \&]
	\& {A/J\oplus J/J^{2}} \\
	{J/J^{2}} \& {T_{A/J}(J/J^{2})_{\leq 1}} \& {A/J} \\
	J \& A \& {A/J}
	\arrow["{=}", from=1-2, to=2-2]
	\arrow["{(\begin{smallmatrix} 1 & 0 \end{smallmatrix})}", from=1-2, to=2-3]
	\arrow["\begin{array}{c} (\begin{smallmatrix} 0 \\ 1 \end{smallmatrix}) \end{array}", from=2-1, to=1-2]
	\arrow["{\sigma_{1}}"', from=2-1, to=3-1]
	\arrow["{\Phi(\sigma_{0},\sigma_{1})_{\leq 1}}", from=2-2, to=3-2]
	\arrow["{=}", from=2-3, to=3-3]
	\arrow[shift left, from=3-1, to=3-2]
	\arrow[shift left, from=3-2, to=3-3]
	\arrow["{\sigma_{0}}", shift left, dashed, from=3-3, to=3-2]
\end{tikzcd}\]
and $\sigma_{0},\sigma_{1}$ being mono, it follows that $\on{Ker}(\Phi(\sigma_{0},\sigma_{1})) \subseteq T_{A/J}(J/J^{2})_{\geq 2}$.
\end{proof}

\begin{corollary} \label{cor:radicalsplitspecies}
    A radical-split algebra $A$ in $\cat{C}$ is Morita equivalent to an admissible quotient of a path algebra of a species.
\end{corollary}

\begin{proof}
  From \autoref{weakGabriel}, we know that $A \simeq (T_{A/J}(J/J^{2}))/I$ for an admissible ideal $I$ in $T_{A/J}(J/J^{2})$. Let $B:=A/J$ and $E=J/J^{2}$. Using \autoref{quasirepresentative}, we may let $(B',E')$ be a species such that $T_{B}(E)$ is Morita equivalent to $T_{B'}(E')$. The Morita equivalence extends to a monoidal equivalence $\functor{F}:\bimod(T_{B}(E), T_{B}(E)) \simeq \bimod(T_{B'}(E'), T_{B'}(E'))$, hence giving a bijection between ideals in $T_{B}(E)$ and ideals in $T_{B'}(E')$, sending $I$ to $\functor{F}(I)$. The equivalence $\functor{F}$ induces further a biequivalence $\on{bimod}(\bimod(T_{B}(E), T_{B}(E))) \simeq \on{bimod}(\bimod(T_{B'}(E'), T_{B'}(E')))$, sending $T_{B}(E)/I$ to $T_{B'}(E')/F(I)$, and giving rise to an
  equivalence $\on{mod}_{\bimod(T_{B}(E), T_{B}(E))}T_{B}(E)/I \simeq \on{mod}_{\bimod(T_{B'}(E'), T_{B'}(E'))}T_{B'}(E')/F(I)$. The result follows by passing under the well-known equivalences 
  \[
  \on{mod}_{\bimod(T_{B'}(E'), T_{B'}(E'))}T_{B'}(E')/F(I) \simeq \on{mod}_{\cat{C}}T_{B'}(E')/F(I)
  \]
  and $\on{mod}_{\bimod(T_{B}(E), T_{B}(E))}T_{B}(E)/I \simeq \on{mod}_{\cat{C}}T_{B}(E)/I$.

  Since the Morita equivalence $F$ of \autoref{quasirepresentative} is given by $N \otimes_{B} - \otimes_{B} M$, for $M \in \bimod(B,B')$ and $N \in \bimod(B',B)$, we find that not only $\functor{F}(T_{B}(E)) = T_{B'}(E')$, but further also $\functor{F}(T_{B}(E)_{m}) = T_{B'}(E')_{m}$ for any $m \in \mathbb{Z}_{\geq 0}$, showing that, since $I$ is an admissible ideal, so is $\functor{F}(I)$, proving that $T_{B'}(E')/\functor{F}(I)$ is an admissible quotient of $T_{B'}(E')$.
\end{proof}

\section{The fusion category setting}\label{sec:fusioncat}
For the remainder of this document, $\cat{D}$ will denote a fusion category (a semisimple finite tensor category).
In this section, we prove results about radicals and semisimple algebras in the setting of fusion categories, including showing that semisimple objects are closed under the $\cat{D}$-action, and reconciling the different notions of ``semisimple'' associated with an algebra. 
Moreover, we associate to every finite $\cat{D}$-module category a directed graph which is an invariant of $\cat{D}$-module categories.
This directed graph will be used at the final section for our main theorems.

\subsection{Radicals of algebras in fusion categories and semisimplicity}
\begin{proposition}\label{mixedfusion}
    The bijection of \autoref{CSZBijection} simplifies in the fusion case: for $A \in \Alg{\cat{D}}$, it sends a $\cat{D}$-stable ideal $\mathfrak{I}$ in $\projda$ to the presheaf $\mathfrak{I}(-,A)$.
\end{proposition}

\begin{proof}
    Since for a fusion category, $\on{id}_{\mathbb{1}}$ is an epimorphism onto $\mathbb{1}$ from a projective object, we may choose it as $q$ in \autoref{CSZBijection} (recall that the construction is independent of the choice of $q$), obtaining the claim as an immediate consequence.
\end{proof}

\begin{proposition}\label{radicalrepresentation}
  The radical $J(A)$ represents the presheaf $\mathfrak{J}_{\cat{D}}(-,A)$.   
\end{proposition}

\begin{proof}
 This follows immediately by combining \autoref{mixedfusion} with \autoref{eqmoduleradical}.
\end{proof}

\begin{corollary}\label{radicalfractions}
For any $m,n \in \mathbb{Z}_{\geq 0}$ with $m\leq n$ and $V \in \mathcal{C}$, we have
\[
 \on{Hom}_{\mathcal{C}}(V,J^{m}/J^{n}) \simeq (\mathfrak{J}_{\cat{D}}^{m}/\mathfrak{J}_{\cat{D}}^{n})_{\on{mod}_{\mathcal{C}}(A)}(V\otimes A, A)
\]
\end{corollary}

\begin{proof}
 Combining \autoref{radicalrepresentation} with the characterization obtained in \cite[Lemma~8.2]{CSZ} of a product $MI$ such as in \autoref{idealtimesbimodule}, we find inductively $\on{Hom}_{\mathcal{C}}(V,J^{m}) \simeq \mathfrak{J}_{\cat{D}}^{m}(V \otimes A, A)$. The result follows by observing that by projectivity of $V$ we have 
\[
 \on{Hom}_{\mathcal{C}}(V,J^{m}/J^{n}) \simeq \mathfrak{J}_{\cat{D}}^{m}(V \otimes A, A)/\mathfrak{J}_{\cat{D}}^{n}(V \otimes A, A) =\mathfrak{J}_{\cat{D}}^{m}/\mathfrak{J}_{\cat{D}}^{n}(V \otimes A, A). 
\]
\end{proof}

\begin{proposition}\label{FusionRadical}
 The $\cat{D}$-radical $\mathfrak{J}_{\cat{D}}$ coincides with the ordinary radical $\mathfrak{J}$ of $\projda$.
\end{proposition}

\begin{proof}
 Let $V,W \in \cat{D}$. Since $\cat{D}$ is semisimple, all objects in $\cat{D}$ are projective, so we have a short exact sequence
\[\begin{tikzcd}[ampersand replacement=\&]
	{\on{Hom}_{\cat{D}}(V, W\lact J)} \& {\on{Hom}_{\cat{D}}(V, W\lact A)} \& {\on{Hom}_{\cat{D}}(V, W\lact A/J)}
	\arrow[hook, from=1-1, to=1-2]
	\arrow[two heads, from=1-2, to=1-3]
\end{tikzcd}\]
Since $\on{Hom}_{\cat{D}}(V, W\lact A/J)$ determines $\on{Hom}_{A/J}(V\lact A/J, W\triangleright A/J)$, and the category of free $A/J$-modules determines that of projective $A/J$-modules, the above sequence together with \autoref{radicalrepresentation} establishes an equivalence $\projda/\mathfrak{J}_{\cat{D}} \simeq \on{proj}_{\cat{D}}(A/J)$.

By \cite[Lemma~6.11]{CSZ}, we have $\mathfrak{J}_{\cat{D}} \subseteq \mathfrak{J}$. But since $\on{mod}_{\cat{D}}(A/J)$ is an exact $\cat{D}$-module category (by \autoref{CSZEO}), and $\cat{D}$ is fusion, we know that $\on{mod}_{\cat{D}}(A/J) = \on{proj}_{\cat{D}}(A/J)$ is semisimple. Thus from the equivalence $\projda/\mathfrak{J}_{\cat{D}} \simeq \on{proj}_{\cat{D}}(A/J)$ we find that $\mathfrak{J} \subseteq \mathfrak{J}_{\cat{D}}$.
\end{proof}

\begin{remark}
    In the characteristic zero case with $\cat{D}$ the category of modules over a semisimple Hopf algebra, the proposition above is shown in \cite[Theorem 3.6]{LMS}
\end{remark}

The following result shows that semisimple objects in a module category are closed under the fusion category action.
\begin{proposition}\label{semisimplestable}
  Given a finite $\cat{D}$-module category $\cat{M} \simeq \on{mod}_{\cat{D}}(A)$, the full subcategory $\mathbb{L}(\cat{M})$ of its semisimple objects forms a $\cat{D}$-module subcategory of $\cat{M}$, equivalent to $\on{mod}_{\cat{D}}(A/J)$.   
  In particular, the following are equivalent:
  \begin{enumerate}
      \item \label{item:Lsemisimple}
        $L \in \cat{M}$ is semisimple.
      \item \label{item:allVLsemisimple}
        $V \lact L \in \cat{M}$ is semisimple for all $V \in \cat{D}$.
      \item \label{item:someVLsemisimple}
        $V \lact L \in \cat{M}$ is semisimple for some $V \in \cat{D}$.
  \end{enumerate} 
\end{proposition}

\begin{proof}
 By \autoref{CSZquotient}, the embedding $\on{res}(\pi)$ of $\on{mod}_{\cat{D}}(A/J)$ in $\on{mod}_{\cat{D}}(A)$ embeds it as subquotients of objects of the form $V\lact L$, for $V \in \cat{C}$ and $L \in \mathbb{L}(\on{mod}_{\cat{D}}(A))$. Since the embedding preserves and reflects exact sequences, a Jordan-Hölder filtration of $V \lact L$ in $\on{mod}_{\cat{D}}(A)$ yields a filtration in $\on{mod}_{\cat{D}}(A/J)$, which is split by semisimplicity of $\on{mod}_{\cat{D}}(A/J)$. This establishes the semisimplicity of $V\lact L$, by $\mathbb{k}$-linearity of $\on{res}(\pi)$.

 We now establish the three equivalences.
 The implication \autoref{item:Lsemisimple} $\implies$ \autoref{item:allVLsemisimple} follows from the previous paragraph, and \autoref{item:allVLsemisimple} $\implies$ \autoref{item:someVLsemisimple} is immediate.
 Now suppose $V \lact L \in \cat{M}$ is semisimple for some $V \in \cat{D}$. Then it follows that $\rightdual{V} \lact (V \lact L)$ is semisimple. 
 Since $\mathbb{1}$ is a summand of $\rightdual{V} \otimes V$, $L$ appears as a summand of the semisimple object $\rightdual{V} \lact (V \lact L) \simeq (\rightdual{V} \otimes V) \lact L$, and so $L$ is also semisimple.
\end{proof}

\begin{remark}
    A different proof of \autoref{FusionRadical} can be given using \autoref{semisimplestable}: the equivalence $\on{proj}_{\cat{D}}(A)/\mathfrak{J}_{\cat{D}} \simeq \mathbb{L}(\on{mod}_{\cat{D}}(A))$ immediately entails $\mathfrak{J}_{\cat{D}} = \mathfrak{J}$.
\end{remark}

The following corollary reconciles the different properties of ``semisimple'' associated to an algebra.
\begin{corollary}\label{cor:reconcilesemisimplealg}
    Let $\cat{D}$ be a fusion category and $A$ be an algebra in $\cat{D}$. The following are equivalent.
    \begin{enumerate}
        \item \label{item:modulecatsemisimple}
            The $\cat{D}$-module category $\modda$ is semisimple.
        \item \label{item:fusionAsemisimplealgebra}
            $A$ is a semisimple algebra (in the sense of \autoref{defn:semisimplealgebra}). 
        \item  \label{item:fusionAsemisimplemodule}
            $A \in \modda$ is a semisimple module over itself.
    \end{enumerate}
\end{corollary}
\begin{proof}
    The first two equivalences are exactly \autoref{CSZEO}, since over fusion categories, exact module categories are semisimple module categories; see \autoref{lem:semisimpleisexact}.
    The implication \autoref{item:modulecatsemisimple} $\implies$ \autoref{item:fusionAsemisimplemodule} is immediate.
    To show \autoref{item:fusionAsemisimplemodule} $\implies$ \autoref{item:modulecatsemisimple}, recall that a semisimple category is an abelian category where all of its projectives are semisimple. 
    Since all projectives in $\modda$ are summands of $X \lact A$ for some $X \in \cat{D}$ and \autoref{semisimplestable} shows that $X \lact A$ is semisimple, our desired result follows.
\end{proof}

\begin{corollary}\label{projcovercompatible}
   Let $X \in \cat{M}$ be an object of a finite $\cat{D}$-module category $\cat{M}$. Let $P_{1} \xrightarrow{p_{1}} P_{0}$ be a minimal projective presentation of $X$.
   
   Then for any $V \in \cat{D}$, we have that $V \lact p_{1}$ is a minimal projective presentation of $V \lact X$. In particular, $P_{0}(V\lact X) = V \lact P_{0}(X)$.
\end{corollary}

\begin{proof}
    It suffices to check that $V \lact p_{1}$ is radical. This follows from $p_{1}$ being radical, together with $\cat{D}$-stability of the radical $\mathfrak{J}(\projectives{\cat{M}})$, which in turn is a consequence of \autoref{FusionRadical}.
\end{proof}

\begin{proposition}
    Taking tops of objects in a finite $\cat{D}$-module category $\cat{M}$ is a $\cat{D}$-module functor:
    \[
     \on{top}(V \lact X) \simeq V \lact \on{top}(X),
    \]
    for all $V \in \cat{D}$ and $X \in \cat{M}$.
\end{proposition}

\begin{proof}
  Since the inclusion $\functor{L}: \mathbb{L}(\cat{M}) \rightarrow \cat{M}$ is a $\cat{D}$-module functor by \autoref{semisimplestable}, by \autoref{adjointmodule} so is its left adjoint $\on{top}$, as described in \autoref{topsocle}. 
\end{proof}

\subsection{Basic semisimples and projective generators}
Recall from \autoref{defn:C-orbitsimples} the equivalence relation $\sim \coloneqq \sim_H = \sim_F$ defined on $\on{Irr}(\cat{M})$ for a finite $\cat{D}$-module category $\cat{M}$.
In the fusion category setting, we have yet another equivalent definition.
\begin{proposition} \label{prop:semisimpleorbitfusion}
  Let $\cat{M}$ be a finite $\cat{D}$-module category. 
  By \autoref{semisimplestable}, the equivalence relation $\sim$ on $\on{Irr}(\cat{M})$ from \autoref{defn:C-orbitsimples} simplifies to the following:
  \[
 L \sim L' \text{ if there is } V \in \cat{D} \text{ such that } L' \text{ is a direct summand of } V\lact L.
  \]
\end{proposition}

\begin{remark}
    Note that if $\cat{M}$ is a $\vect_G$-module category; equivalently $\cat{M}$ has a (strong) categorical action of $G$, then $L \sim L'$ if and only if $L$ and $L'$ are in the same $G$-orbit. In particular, $\on{Irr}(\cat{M})/\sim$ is exactly the set of $G$-orbits of the simple objects.
\end{remark}
\begin{definition}
    We say that a semisimple object $L \in \mathbb{L}(\cat{M})$ is {\it $\cat{D}$-basic} if for any two different simple summands $L_{1},L_{2}$ of $L$, we have $L_{1} \not\sim L_{2}$.
    We say that a projective object $P \in \projectives{\cat{M}}$ is {\it $\cat{D}$-basic} if $\on{top}(P)$ is $\cat{D}$-basic in $\mathbb{L}(\cat{M})$.
    Accordingly, an algebra $A$ in $\cat{D}$ is \emph{$\cat{D}$-basic} if $A$ is a $\cat{D}$-basic as a projective module in $\modda$.
\end{definition}

\begin{remark}
    Let $\on{indproj}(\cat{M})$ denote a set of representatives for the indecomposable projectives in $\cat{M}$.
    It is well known that $P \in \on{indproj}(\cat{M}) \mapsto \on{top}(A) \in \on{Irr}(\cat{M})$ is a bijection.
    If we define the following a similar equivalence relation on $\on{indproj}(\cat{M})$: $P \sim_p P'$ if and only if $P$ is a direct summand of $V \lact P$ for some $V \in \cat{D}$, the bijection above descends to a bijection between $\on{indproj}(\cat{M})/\sim_p$ and $\on{Irr}(\cat{M})/\sim$.
    Moreover, a (not necessarily indecomposable) projective $P$ is $\cat{D}$-basic if and only if for any two different indecomposable summands $P_{1},P_{2}$ of $P$, we have $P_{1} \not\sim_p P_{2}$.
\end{remark}

\begin{proposition}\label{dbasic}
   If $L \in \mathbb{L}(\cat{M})$ is $\cat{D}$-basic, then $\underline{\on{End}}_{\cat{M}}(L)$ is a finite direct product of division algebras $\prod_{i=1}^r \underline{\on{End}}_{\cat{M}}(L_i)$, with $L_i$ simple objects and $L = \bigoplus_{i=1}^r L_i$. As such, if $P \in \projectives{\cat{M}}$ is $\cat{D}$-basic, then $\underline{\on{End}}_{\cat{M}}(\on{top}(P))$ is a finite direct product of division algebras. 
\end{proposition}
\begin{proof}
    Both of the claims are an immediate consequence of \autoref{endproduct}.
\end{proof}

\begin{lemma}\label{topaj}
    There is a right $A$-module isomorphism $A/J \simeq \on{top}(A)$.
\end{lemma}
\begin{proof}
Let $L$ be an $A/J$-module, hence also naturally an $A$-module through the quotient algebra morphism $A \twoheadrightarrow A/J$. We have a chain of isomorphisms natural in $L$:
\[
 \on{Hom}_{A/J}(A/J, L) \simeq \on{Hom}_{\cat{C}}(\mathbb{1},L) \simeq \on{Hom}_{A}(A,L) \simeq \on{Hom}_{\mathbb{L}(A)}(\on{top}(A), L) \simeq \on{Hom}_{A/J}(\on{top}(A), L),
\]
the first coming from the free-forgetful adjunction for $A/J$, the second from the free-forgetful adjunction for $A$, the third by the adjunction of \autoref{topsocle} and the last by \autoref{semisimplestable}.
By Yoneda lemma it follows that $A/J \simeq \on{top}(A)$ in $\modv{A/J}$, and hence we also have an isomorphism after passing with the inclusion $\modv{A/J} \hookrightarrow \modda$.
\end{proof}

\begin{proposition}\label{prop:basicalgebra}
    If $A$ is a $\cat{D}$-basic algebra, then $A/J$ is a product of division algebras $\prod_{i=1}^r \ienda{L_i}$, where $L_i$ are simple $A$-modules such that $A/J \simeq \bigoplus_{i=1}^r L_i \in \modda$.
\end{proposition}
\begin{proof}
    Since $A/J \simeq \on{top}(A)$ as $A$-modules by \autoref{topaj}, we have the following chain of isomorphisms of algebras
    \[
    A/J \simeq \underline{\on{End}}_{A/J}(A/J) \simeq \ienda{A/J} \simeq \ienda{\on{top}(A)},
    \]
    where the middle isomorphism a consequence of \autoref{semisimplestable} and \autoref{enrichedhom}.
    Now $A$ is $\cat{D}$-basic means $\on{top}(A) \in \mathbb{L}(\modda)$ is $\cat{D}$-basic, and so the fact that $A/J$ is a product of division algebras just follows from \autoref{dbasic}.
\end{proof}

\begin{proposition} \label{prop:Mtobasicalg}
    Let $\cat{M}$ be a finite module category over $\cat{D}$. 
    Then there exists a $\cat{D}$-basic algebra $A$ such that $\mathcal{M} \simeq \modda$ as $\cat{D}$-module categories.
\end{proposition}
\begin{proof}
   Choose a set of representatives $\{L_1,\ldots, L_r\}$ for $\on{Irr}(\cat{M})$.
   Then $L \coloneqq \bigoplus_{i=1}^r L_i$ is $\cat{D}$-basic by construction, and its projective cover $P_0 \twoheadrightarrow \on{top}(P_0) \simeq L$ is therefore $\cat{D}$-basic by definition.
   Since $\{L_1,\ldots, L_r\}$ is a set of representatives for $\on{Irr}(\cat{M})$, there must exist some $X \in \cat{D}$ such that every simple in $\cat{M}$ appears as a summand of $X\lact L$ (noting that $X\lact L$ is semisimple by \autoref{prop:semisimpleorbitfusion}).
   Since the functor $X \lact -$ is exact on $\cat{M}$ (\autoref{cor:exactonboth}), it follows that $X\lact P_0 \twoheadrightarrow X \lact L$, and so $X\lact P_0$ surjects onto all projective covers of simple objects in $\cat{M}$.
   This shows that $P_0$ is also a $\cat{D}$-projective generator.
   By \autoref{EGNOreconstruction}, we obtain $\cat{M} \simeq \modda$ with $A \coloneqq \iend{P_0}$ a $\cat{D}$-basic algebra, as required.
\end{proof}

\begin{definition}
    Given $S,L \in \mathbb{L}(A)$ with projective covers $P_0(S),P_0(L)$ respectively, we define the object $(J/J^{2})(S,L)$ of $\cat{D}$ as representing the presheaf sending $V \in \cat{D}$ to $(\mathfrak{J}/\mathfrak{J}^{2})_{A}(V\lact P_{0}(S), P_{0}(L))$.

    In particular we have $J/J^{2} = J(A)/J^{2}(A) = (J/J^{2})(\on{top}(A), \on{top}(A))$.
\end{definition}

Clearly, we have the following:
\begin{lemma}
For semisimple objects $S,S',L,L' \in \mathbb{L}(A)$ we have
\[
(J/J^{2})(S\oplus S',L) \simeq (J/J^{2})(S,L) \oplus (J/J^{2})(S',L)\text{ and } 
(J/J^{2})(S,L\oplus L') \simeq (J/J^{2})(S,L) \oplus (J/J^{2})(S,L').
\]

Hence, for a decomposition $\on{top}(A) \simeq \bigoplus_{i=1}^{m} L_{i}$ into simple right $A/J$-modules, we find a direct sum decomposition
\begin{equation}\label{directsumjj2}
    J/J^{2} = J/J^{2}(\on{top}(A), \on{top}(A)) \simeq \bigoplus_{i,j=1}^{m} J/J^{2}(L_{i},L_{j}).
\end{equation}
\end{lemma}

Let $A$ be $\cat{D}$-basic algebra, let $ \on{top}(A) \simeq \bigoplus_{i=1}^{m} L_{i}$ be a decomposition into simple right $A/J$-modules.
Then $A/J \simeq \prod_{i=1}^{m} D_{i}$ is a direct product decomposition of algebras obtained from  \autoref{prop:basicalgebra}, where $D_{i} := \underline{\on{End}}_{A/J}(L_{i})$.
It follows from \autoref{modproduct} that the $A/J$-$A/J$-bimodule $J/J^2$ has a canonical splitting induced by the product algebra decomposition of $A/J \simeq \prod_{i=1}^n D_i$.
We now show that this agrees with the decomposition in \autoref{directsumjj2}.
\begin{proposition} \label{prop:J/J2splitting}
     In the above setting, assume that the morphism $\pi_{0}: A \twoheadrightarrow A/J$ of algebra objects admits a section. Then the $A/J$-$A/J$-bimodule $J/J^{2}$ satisfies
    \[
    D_{i}(J/J^{2})D_{j} \subseteq J/J^{2}(L_{j},L_{i}).
    \]
\end{proposition}
\begin{proof}
    Recall from \cite[Section~2]{CSZ} that the action of $A$ on $J(A)$ can be described in terms of Day convolution in the category of finite-dimensional presheaves on $\cat{D}$. 
    
    Since $\cat{D}$ is semisimple, the restricted Yoneda embedding $\cat{D} \rightarrow [\cat{D}^{\on{op}},\mathbf{vec}]$ is in fact a monoidal equivalence, and the Day convolution $\euler{P}\ast \euler{Q}$ of two presheaves $\euler{P,Q}$ can be described as sending a simple object $U \in \Irr(\cat{D})$ to $\bigoplus_{V,W \in \Irr(\cat{D})} \euler{P}(V) \otimes_{\mathbb{k}} \euler{Q}(W)$.
    
    Thus, as in \cite{CSZ}, the right $A$-action is given by the collection of maps 
    \begin{equation}\label{convolution}
    \Big(\mathfrak{J}(V \lact A, A) \otimes \on{Hom}_{\cat{D}}(W,A) \xiso \mathfrak{J}(V \lact A, A) \otimes \on{Hom}_{-A}(W\lact A,A) \rightarrow \mathfrak{J}(V\otimes W\lact A,A)\Big)_{V,W \in \cat{D}},
    \end{equation}
    where the former map is induced from the canonical isomorphism $\on{Hom}_{\cat{D}}(W,A) \simeq \on{Hom}_{-A}(W\lact A,A)$, and the latter sends $f \otimes g$ to $f \circ V\lact g$ in $\modda$. The left action can be described similarly. 

    The actions of $A$ on $J^{2}$ and $J/J^{2}$ can be described similarly, and the action of $A/J$ on $J$ (and by extension on $J^{2}$ and $J/J^{2}$) obtained by restriction along the morphism $A/J \xrightarrow{\sigma_{0}} A$ is given by precomposing the maps of \autoref{convolution} with the maps
    \[
    \mathfrak{J}(V \lact A, A) \otimes \on{Hom}_{\cat{D}}(W,A/J) \xrightarrow{\on{id} \otimes \on{Hom}_{\cat{D}}(W,\sigma_{0})} 
    \mathfrak{J}(V \lact A, A) \otimes \on{Hom}_{\cat{D}}(W,A).
    \]
    Similarly, the action of $D_{i} \simeq \underline{\on{End}}_{-A}(L_{i})$ is obtained by further precomposing with 
    \[
    \mathfrak{J}(V \lact A, A) \otimes \on{Hom}_{\cat{D}}(W,D_{i}) \xrightarrow{\on{id} \otimes \on{Hom}_{\cat{D}}(W,\iota)} 
    \mathfrak{J}(V \lact A, A) \otimes \on{Hom}_{\cat{D}}(W,A).
    \]
     Again, similar considerations apply to $J^{2}$ and $J/J^{2}$. We thus obtain a commutative diagram
\[\begin{tikzcd}[column sep=small]
	{\mathfrak{J}/\mathfrak{J}^{2}(V\otimes A,A)\otimes \on{Hom}_{-A}(W\lact A, A)} && {\mathfrak{J}/\mathfrak{J}^{2}(V\otimes A,A)\otimes \on{Hom}_{-A}(W\lact P_{0}(L_{i}), P_{0}(L_{i}))} \\
	&& {\mathfrak{J}/\mathfrak{J}^{2}(V\otimes A, A) \otimes (\on{Hom}/\mathfrak{J})_{-A}(W\lact P_{0}(L_{i}),P_{0}(L_{i}))} \\
	{\mathfrak{J}/\mathfrak{J}^{2}(V\otimes W \lact A, A)} && {\mathfrak{J}/\mathfrak{J}^{2}(V\otimes A, A) \otimes \on{Hom}_{\cat{D}}(W,L_{i})}
	\arrow[from=1-1, to=3-1]
	\arrow[from=1-3, to=1-1]
	\arrow[shift left=3, from=1-3, to=2-3]
	\arrow[from=2-3, to=1-3]
	\arrow["\simeq"', from=3-3, to=2-3]
	\arrow[from=3-3, to=3-1]
\end{tikzcd}\]
 The adjacent pair of vertical arrows are given by $\pi_{0}: A \twoheadrightarrow A/J$ and a splitting thereof.
 The lower horizontal arrow describes the right $D_{i}$-action on $J/J^{2}$, and its image describes the object $(J/J^{2})D_{i}$. 
 Since the leftmost vertical map is just composition, this image is clearly contained in ${\mathfrak{J}/\mathfrak{J}^{2}(V\otimes W \lact P_{0}(L_{i}), A)}$, from which it follows that $(J/J^{2})D_{i} \subseteq J/J^{2}(L_{i},\on{top}(A))$. The claim for the left action follows similarly.
\end{proof}

\begin{corollary}\label{arrowsspecify}
  Let $A$ be a $\cat{D}$-basic algebra and let $\on{top}(A) \simeq \bigoplus_{i=1}^{m} L_{i}$ be a decomposition into simple $A/J$-modules. The direct sum decomposition of the $A/J$-$A/J$-bimodule $J/J^{2}$ is given by
  \[
  J/J^2 \simeq \bigoplus_{i,j=1}^{m} J/J^{2}(L_{i},L_{j}),
  \]
  with the $A/J$-$A/J$-bimodule structure on $J/J^{2}(L_{i},L_{j})$ given in \autoref{prop:J/J2splitting}.
  In particular, the associated $D_{i}$-$D_{j}$-bimodule is $J/J^{2}(L_{j},L_{i})$
\end{corollary}

\begin{remark}
   Much like in the classical case, the bimodule of arrows from $i$ to $j$ is given by the object $J/J^{2}(L_{j},L_{i})$, which is the internal version of the classical $(\on{Rad}/\on{Rad}^{2})(P(j),P(i))$.
\end{remark}

\subsection{Internal Ext and directed graphs of module categories}
Recall that for finite abelian categories, one can associated to it a directed graph (which is sometimes called the Ext-quiver).
Here, we will introduce a similar variant for $\cat{D}$-module categories.

Let us first recall the classical setting.
Given a vector space $V$, we shall denote its dual space by $\mathbb{D}V$ -- this is to distinguish the $\mathbb{k}$-linear duality from the duality of an abstract rigid monoidal category $\cat{D}$. For a presheaf $\functor{P}: \cat{A} \rightarrow \vect$, we denote by $\mathbb{D}\functor{P}$ the copresheaf sending $X$ to $\mathbb{D}\functor{P}(X)$, and similarly for copresheaves and profunctors.
The following well-known result gives an effective method of computing the directed graph of a finite abelian category:
\begin{proposition}[{\cite[Proposition~2.4.3]{Ben}, \cite[Lemma~III.2.12]{ASS}}]\label{extradicaliso}
  Let $\cat{A}$ be a finite abelian category and let $L,L' \in \mathbb{L}(\cat{A})$ be semisimple objects thereof. Then there is an isomorphism
  \[
   \mathfrak{J}/\mathfrak{J}^{2}(P_{0}(L), P_{0}(L')) \simeq \mathbb{D}\on{Ext}^{1}_{\cat{A}}(L',L),
  \]
  natural in $L,L'$.
\end{proposition}

We shall upgrade this to the general fusion category setting.
We first define the notion of an internal Ext:
\begin{definition}\label{internalextdef}
   Let $\cat{D}$ be a fusion category and $\cat{M}$ a finite $\cat{D}$-module category. For $X,Y \in \cat{M}$ and $n \geq 0$, the \emph{internal Ext}, denoted by $\iExt_{\cat{M}}^{n}(X,Y)$, is defined as the object in $\cat{D}$ representing the presheaf
   $
   \on{Ext}_{\mathcal{M}}^{n}(-\triangleright X, Y)
   $. In other words, there are isomorphisms natural in $V$:
    \[
    \on{Hom}_{\mathcal{D}}(V, \underline{\on{Ext}}_{\mathcal{M}}^{n}(X,Y))\simeq \on{Ext}_{\mathcal{M}}^{n}(V\triangleright X, Y).
    \]
   Throughout this document, we shall simply denote $\iExt^1_\cat{M}(-,-)$ by $\iExt_\cat{M}(-,-)$.
\end{definition}

As in \autoref{eqn:internalhomformula}, the internal Ext can be easily computed as follows:
\begin{equation}\label{eqn:internalextformula}
    \iExt^i(X,Y) \simeq \bigoplus_{L \in \on{Irr}(\cat{D})} L^{\oplus e_L}
\end{equation}
where $e_L \coloneqq \dim_{\mathbb{k}} \Ext^i(L\lact X, Y)$.

\begin{remark}
    As will be shown in a sequel, the discrete complex $\underline{\on{Ext}}_{\cat{M}}^{\bullet}(X,Y)$ is in fact the internal $\on{Hom}$ for the derived $\euler{D}(\cat{D})$-module category $\euler{D}(\cat{M})$, and equivalently also the derived functor of $\underline{\on{Hom}}_{\cat{M}}(-,-)$. However, this is not necessarily the case for non-semisimple tensor categories $\cat{C}$, where the internal $\on{Hom}$ in $\euler{D}(\cat{C})$ may have a non-zero differential.
\end{remark}

\begin{proposition}\label{extarrowsextra}
 Let $\cat{D}$ be a fusion category.
 For an algebra $A \in \cat{D}$ and $S,L \in \mathbb{L}(A)$ there is an isomorphism 
 \[
 J/J^{2}(S,L) \simeq \leftdual{\big(\underline{\on{Ext}}^{1}_{-A}(L, S)\big)}.
 \]
\end{proposition}

\begin{proof}
  We show $\underline{\on{Ext}}^{1}_{-A}(L, S) \simeq \rightdual{(J/J^{2})}$, by showing 
  that the two objects define the same (representable) presheaf. We have  
  \begin{equation}\label{objectiso}
 \begin{aligned}
  &\on{Hom}_{\cat{D}}(-,\underline{\on{Ext}}_{-A}^{1}(L, S))\simeq \on{Ext}_{-A}^{1}(-\triangleright L, S)
  \\
  &\simeq \mathbb{D}\mathfrak{J}/\mathfrak{J}^{2}(P_{0}(S),P_{0}(-\triangleright L))\\
  &\simeq \mathbb{D}\mathfrak{J}/\mathfrak{J}^{2}(P_{0}(S),-\triangleright P_{0}(L))
  \simeq \mathbb{D}\mathfrak{J}/\mathfrak{J}^{2}(\leftdual{(-)}\triangleright P_{0}(S),P_{0}(L))\\
  &\simeq \mathbb{D}\on{Hom}_{\cat{D}}(\leftdual{(-)},(J/J^{2})(S,L))\\
  &\simeq \mathbb{D}\on{Hom}_{\cat{D}}(\rightdual{(J/J^{2})(S,L)},-) \simeq \on{Hom}_{\cat{D}}(-,\rightdual{(J/J^{2})(S,L)}).
 \end{aligned}
\end{equation}
Here, the first isomorphism is \autoref{internalextdef}. The second is the isomorphism of \autoref{extradicaliso}. The third follows from $P_{0}(-\lact L) \simeq -\lact P_{0}(L)$, which is an instance of \autoref{projcovercompatible}. 
The fourth is an instance of the $\on{Hom}$-isomorphism from the adjunction $\leftdual{-} \lact \dashv - \lact$, which preserves radical morphisms since $\mathfrak{J}$ is $\cat{D}$-stable by \autoref{FusionRadical}. The fifth is by \autoref{radicalfractions}, the sixth by rigidity of $\cat{D}$. Finally, the seventh isomorphism comes from that for the Nakayama functor $\nu_{\cat{D}}$ for $\cat{D}$:
\[
 \on{Hom}_{\cat{D}}(-,\nu(P)) \simeq \mathbb{D}\on{Hom}_{\cat{D}}(P,-) \text{ for } P \in \projectives{\cat{D}},
\]
 but since $\cat{D}$ is semisimple, we have $\nu_{\cat{D}} \simeq \on{Id}_{\cat{D}}$.

 Naturality of all of these isomorphisms is automatically obtained from \autoref{unnaturality}.
\end{proof}

We emphasise that the isomorphism in the proposition above is an isomorphism of objects of $\cat{D}$, rather than $\iend{L}$-$\iend{S}$-bimodules. Indeed, in order to equip $\leftdual{\iExt_{\cat{M}}(L,S)}$ with such a bimodule structure, it would be necessary to assume $\cat{D}$ to be pivotal ($\iExt_{\cat{M}}(L,S)$ is naturally a $\iend{S}$-$\iend{L}$-bimodules).
It remains an open question whether every fusion category can be equipped with a pivotal structure \cite[Question 4.8.3]{EGNO}. Nonetheless, observe that \autoref{extarrowsextra} is analogous to the often-used variant of \autoref{extradicaliso}, which identifies only the dimensions of spaces, rather than $\on{End}(L')$-$\on{End}(L)$-bimodules.

We can now associate to every finite $\cat{D}$-module category a directed graph as follows.
\begin{definition} \label{defn:internalExtspecies}
    Let $\cat{M}$ be a finite $\cat{D}$-module category.
    The \emph{directed graph of $\cat{M}$} is defined as the directed graph with set of vertices given by $\on{Irr}(\cat{M})/\sim$, and for each $[L], [L'] \in \on{Irr}(\cat{M})/\sim$, we have an arrow $[L] \to [L']$ if and only if $\iExt_{\cat{M}}(L,L') \neq 0$; this is well-defined irrespective of the representatives $L \in [L]$ and $L' \in [L']$, as we will show in \autoref{directedgraphindependent}.
\end{definition}

\begin{remark}
    If $\cat{D}$ is moreover pivotal, one could define a $\cat{D}$-species for each set of representatives $\{L_1, \ldots, L_n\}$ of $\on{Irr}(\cat{M})/\sim$ as follows, which one might like to call it the \emph{internal Ext-species}. It has underlying directed graph given by the directed graph of $\cat{M}$, where each vertex is labelled by $D_i \coloneqq \iend{L_i}$ and each arrow is labelled by $E_{ij} \coloneqq \leftdual{\iExt_{\cat{M}}(L_i,L_j)}$; the $(\iend{L_i},\iend{L_j})$-bimodule structure on $E_{ij}$ is defined using the pivotal structure.
\end{remark}

\begin{lemma}\label{directedgraphindependent}
  Let $[L_1], [L_2]$ be two equivalence classes in $\on{Irr}(\cat{M})$.
  Then $\iExt_{\cat{M}}(L_1,L_2) \neq 0$ if and only if $\iExt_{\cat{M}}(L'_1,L'_2) \neq 0$ for any representatives $L'_1\in[L_1]$ and $L'_2 \in [L_2]$.
\end{lemma}
\begin{proof}
  Suppose $\underline{\on{Ext}}_{\cat{M}}(L_{1},L_{2}) \neq 0$, so there exists $V \in \cat{D}$ such that $\on{Hom}_{\cat{D}}(V, \underline{\on{Ext}}_{\cat{M}}(L_{1},L_{2})) \neq 0$. 
  Since $L_k \sim L'_k$ for all $k$, using \autoref{prop:semisimpleorbitfusion} we get for each $k$ an object $W_{k} \in \cat{D}$ such that $L_{k}$ is a direct summand of $W_{k} \lact L'_{k}$.
  Then
  \[
  \on{Hom}_{\cat{D}}(V, \underline{\on{Ext}}_{\cat{M}}(L_{1},L_{2})) \simeq \on{Ext}_{\cat{M}}(V\lact L_{1},L_{2})
  \]
  is a direct summand of $\on{Ext}_{\cat{M}}((V\otimes W_{1})\lact L'_{1},W_{2}\lact L'_{2})$. The adjunction $\leftdual{W_{2}}\lact - \dashv W_{2}\lact -$ of exact functors yields an isomorphism 
  \[
  \on{Ext}_{\cat{M}}((V\otimes W_{1})\lact L'_{1},W_{2}\lact L'_{2}) \simeq \on{Ext}_{\cat{M}}((\leftdual{W_{2}} \otimes V \otimes W_{1}) \lact L'_{1}, L'_{2}), 
  \]
  and by \autoref{internalextdef}, this is isomorphic to $\on{Hom}_{\cat{D}}(\leftdual{W_{2}} \otimes V \otimes W_{1}, \underline{\on{Ext}}_{\cat{M}}(L'_{1},L'_{2}))$, showing that $\underline{\on{Ext}}_{\cat{M}}(L'_{1},L'_{2})$ is non-zero.
  The other implication direction is immediate.
\end{proof}

When $\cat{M}$ is defined from a $\cat{D}$-species, we have the following result.
\begin{theorem}\label{directedgraphforspecies}
    Let $A$ be a quotient of the path algebra of a $\cat{D}$-species $(I, (D_i)_{i \in I}, (E_{ij})_{i,j\in I})$, and let $\cat{M} \coloneqq \modda$.
    The directed graph of $\cat{M}$ agrees with the directed graph of the $\cat{D}$-species.
\end{theorem}
\begin{proof}
   This follows immediately by applying \autoref{extarrowsextra} to $S = L_{i}$ and $L= L_{j}$ for $i,j=1,\ldots,n$ given a set of representatives $\{L_1, \ldots, L_n\}$ of $\on{Irr}(\cat{M})/\sim$.
\end{proof}

\section{Main theorems}\label{mainresult}
For this section, we follow \cite[Definition 2.5.8 \& Corollary 2.5.9]{DSS} and say that a fusion category $\cat{D}$ is \emph{separable} if its Drinfeld centre $\cat{Z(D)}$ (\cite[Definition 7.13.1]{EGNO}) of $\cat{D}$ is semisimple, i.e.\ $\cat{Z(D)}$ is also a fusion category. 
We remark that this is a considerably mild condition, which holds, for example, when $\mathbb{k}$ has characteristic 0 \cite[Theorem 9.3.2]{EGNO}.
More importantly, it will also be a necessary condition for our main theorems; see \autoref{rmk:SemisimpleCentreNecessary}.
One of the main points is that under this assumption, semisimple algebras and separable algebras are the same, which allows us to deduce that all algebras are radical-split.
Through the directed graph of a finite $\cat{D}$-module category, we also deduce that any hereditary finite $\cat{D}$-module category is equivalent to the category of modules over the path algebra of an acyclic $\cat{D}$-species (the converse always hold).

\subsection{Gabriel's theorem for algebras in fusion categories}

\begin{definition}
    An algebra $A$ in $\cat{D}$ is said to be {\it separable} if the multiplication $\mu: A \otimes A \rightarrow A$ admits a section in $\bimod(A,A)$.
\end{definition}

\begin{remark} \label{rmk:FTCseparable}
    The reason for not considering separable algebras in the general setting of finite tensor categories is that, while the definition can be repeated verbatim in that case, separability is not Morita-invariant in finite tensor categories -- for example, while the unit object $\mathbb{1}$ of a finite tensor category $\cat{C}$ is always a separable algebra, if $\cat{C}$ is not semisimple then the algebra $\underline{\on{Hom}}(P,P)$, for any $P \in \projectives{\cat{C}}$, Morita equivalent to $\mathbb{1}$, is not separable.
\end{remark}

\begin{proposition}[\protect{\cite[Theorem~1.30]{AMS}}] \label{prop:separableimpliessemisimple}
    A separable algebra is semisimple (in the sense of \autoref{defn:semisimplealgebra}).
\end{proposition}
\begin{remark}
    Under the equivalences in \autoref{cor:reconcilesemisimplealg}, this result is also shown in \cite[Proposition~2]{Os}.
\end{remark}
The following is a consequence of \cite[Theorem~3.8]{AMS}:
\begin{proposition} \label{prop:separablesection}
 Let $A$ be an algebra in $\cat{D}$ and let $I$ be a nilpotent ideal in $A$ such that $A/I$ is separable. Then there is an algebra morphism $\sigma: A/I \rightarrow A$ such that $\pi \circ \sigma = \on{id}_{A/I}$, where $\pi$ is the projection from $A$ onto $A/I$.
\end{proposition}

\begin{proof}
 Since $I$ is nilpotent, the limit of the diagram
\[\begin{tikzcd}[ampersand replacement=\&,sep=small]
	\cdots \& {A/I^{n+1}} \& {A/I^{n}} \& {A/I^{n-1}} \& \cdots \& {A/I}
	\arrow[from=1-1, to=1-2]
	\arrow[from=1-2, to=1-3]
	\arrow[from=1-3, to=1-4]
	\arrow[from=1-4, to=1-5]
	\arrow[from=1-5, to=1-6]
\end{tikzcd}\]
is isomorphic to that of its truncation $A \rightarrow A/I^{d-1} \cdots \rightarrow A/I$, where $d = \min\setj{m \; | \; I^{m} = 0}$. The limit of the truncation is $A$, since the truncation is a diagram with an initial vertex. Since $\cat{D}$ is semisimple, the maps in the above diagram are split. Hence, by 
\cite[Lemma~3.4]{AMS}, the maps in the diagram define Hochschild extensions. 
Since $A/I$ is separable, by \cite[Corollary~3.10]{AMS}, it satisfies the equivalent conditions of \cite[Theorem~3.8]{AMS}. The result follows directly from \cite[Theorem~3.8(c)]{AMS}.
\end{proof}

The following lemma is essentially the same as the proof of \cite[Corollary 9.3.3]{EGNO}.
\begin{lemma} \label{lem:dualfusion}
    Let $\cat{D}$ be a separable fusion category.
    Then the category of bimodules $\bimodd(A,A)$ for any semisimple algebra $A$ is a multifusion category (it is fusion if $\modda$ is indecomposable as a $\cat{D}$-module category).
    In particular, $\bimodd(A,A)$ is semisimple.
\end{lemma}
\begin{proof}
    Recall from \autoref{cor:reconcilesemisimplealg} that $A$ being semisimple is equivalent to $\modda$ being semisimple (i.e.\ exact over $\cat{D}$; see \autoref{lem:semisimpleisexact}). By \cite[\S 7.12]{EGNO}, $\bimodd(A,A)$ is a finite multitensor category -- it is the (monoidal-opposite of the) dual tensor category $\cat{D}^*_{\cat{M}}$ of $\cat{D}$ with respect to $\cat{M} = \modda$.
    We just have to argue that it is moreover semisimple.
    
    Note that the forgetful functor $\cat{Z}(\cat{D}^{\otimes op}) \simeq \cat{Z(\bimodd(A,A))} \to \bimodd(A,A)$ sends the monoidal unit to the monoidal unit, and moreover, it sends projectives to projectives.
    If $\cat{Z(D)}$ is semisimple (and so is $\cat{Z}(\cat{D}^{\otimes op})$), then the monoidal unit of $\cat{Z}(\cat{D}^{\otimes op})$ is projective, which implies that the monoidal unit of $\bimodd(A,A)$ is also projective. This shows that $\bimodd(A,A)$ is semisimple, as required.
\end{proof}

The following result is well-known to experts; see e.g.\ \cite[Proposition 2.5.10]{DSS} and \cite[Corollary~5.4]{KZ}. We include a proof for completeness.
\begin{proposition}\label{KongZheng}
Let $\cat{D}$ be a separable fusion category and $A \in \cat{D}$ an algebra.
Then the following are equivalent:
\begin{enumerate}
    \item\label{item:semisimple} $A$ is semisimple;
    \item\label{item:separable} $A$ is a separable algebra.
\end{enumerate}
In particular, in a separable fusion category, an algebra being separable is equivalent to any of the conditions stated in \autoref{cor:reconcilesemisimplealg}.
\end{proposition}
\begin{proof}
    The direction \ref{item:separable} $\implies$ \ref{item:semisimple} is \autoref{prop:separableimpliessemisimple} (and does not require semisimplicity of $\cat{Z}(\cat{D})$).
    
    On the other hand, let $A$ be semisimple.
    If the center $\cat{Z}(\cat{D})$ is semisimple, then $\bimod(A,A)$ is also semisimple by \autoref{lem:dualfusion}.
    In particular, the multiplication map $\mu: A\otimes A \to A$, which is a morphism of $(A,A)$-bimodules, has to split.
\end{proof}

\begin{remark}
    In fact, many of the forthcoming results only require that semisimple algebras and separable algebras are the same. While $\cat{D}$ being separable is sufficient (as shown above), we do not know if it is necessary for semisimple algebras to be (Morita equivalent to) separable algebras.
    See \autoref{eg:semisimplecentre} for an example of a semisimple algebra that is not separable (this algebra is in a fusion category that is not separable).
\end{remark}

Recall that given an algebra $A$, we denote its Jacobson radical by $J$.
\begin{corollary}[Wedderburn--Malcev theorem] \label{cor:WedderburnMalvecthm}
Let $A$ be an algebra in a separable fusion category $\cat{D}$. Then the algebra morphism $\pi_{0}: A \twoheadrightarrow A/J$ admits a section in $\mathrm{Alg}({\cat{D}})$.
\end{corollary}
\begin{proof}
    By \autoref{CSZEO}, the algebra $A/J$ is semisimple. Since we assume $\cat{Z}(\cat{D})$ to be semisimple, by \autoref{KongZheng} we find that $A/J$ is separable. Finally, from \autoref{CSZEO} we know that $J$ is nilpotent. 

    Thus $J$ is a nilpotent ideal such that $A/J$ is separable, and thus by \autoref{prop:separablesection}, the projection $\pi_{0}$ admits a section. 
\end{proof}

\begin{proposition}\label{radicalsplitness}
   Any algebra in a separable fusion category $\cat{D}$ is radical-split.
\end{proposition}
\begin{proof}
    By the corollary above, it remains to show that $J \twoheadrightarrow J/J^{2}$ admits a section in $\bimodd(A/J, A/J)$. This follows from semisimplicity of $\bimodd(A/J, A/J)$ by \autoref{lem:dualfusion}  (see also \cite[Theorem~1.30]{AMS}).
\end{proof}

The following is our first main theorem; cf.\ \autoref{weakGabriel} and \autoref{cor:radicalsplitspecies}.
\begin{theorem}\label{thm:mainthm}
   Let $\cat{D}$ be a separable fusion category and let $A$ be a $\cat{D}$-basic algebra object in $\cat{D}$ with $A/J \simeq \bigoplus_{i\in I} L_i$ in $\modda$.
   
   Then $A$ is isomorphic to an admissible quotient of the path algebra of a $\cat{D}$-species $(I, (D_i)_{i\in I}, (E_{ij})_{i,j\in I})$, where $D_i \coloneq \underline{\on{End}}_{A}{L_i}$ and $E_{ij} \coloneq J/J^2(L_j,L_i)$.
\end{theorem}
\begin{proof}
    This first statement is an immediate consequence of \autoref{radicalsplitness} and \autoref{weakGabriel}, where by \autoref{prop:basicalgebra} the $\cat{D}$-basic assumption on $A$ ensures that $A/J$ is a finite product of division algebras $\iend{L_i}$, so that we obtain a $\cat{D}$-species (and not just a quasi-species). The equation $E_{ij} = J/J^2(L_j,L_i)$ is a direct consequence of \autoref{arrowsspecify}.
\end{proof}

For general finite $\cat{D}$-module category $\cat{M}$, we have the following.
\begin{corollary}\label{extspecies}
   Let $\cat{D}$ be a separable fusion category and $\cat{M}$ be a finite $\cat{D}$-module category.
   Choose a set of representatives $\{L_1,\ldots, L_r\}$ for $\on{Irr}(\cat{M})/\sim$.
   Then $\cat{M}$ is equivalent to $\modda$ for $A$ an admissible quotient of the path algebra of a $\cat{D}$-species given by $(I, (D_i)_{1\leq i \leq r}, (E_{ij})_{1\leq i,j \leq r})$, with $D_i \coloneqq \iend{L_i}$ and $E_{ij} \simeq \leftdual{\iExt(L_i,L_j)}$ as objects in $\cat{D}$ for all $i,j \in \{1,\ldots,r\}$.
\end{corollary}
\begin{proof}
  Let $L \coloneqq \bigoplus_{i=1}^{r} L_{i}$, $P_0$ be its projective cover and $A \coloneqq \underline{\on{End}}_{\cat{M}}(P_0)$.
  As described in the proof of \autoref{prop:Mtobasicalg}, we have $\cat{M} \simeq \modda$ with $A$ a $\cat{D}$-basic algebra.
  By \autoref{thm:mainthm}, $A$ is isomorphic to an admissible quotient of the path algebra of a $\cat{D}$-species.
  That $E_{ij} \simeq \leftdual{\iExt(L_i,L_j)}$ follows from \autoref{extarrowsextra}.
\end{proof}

\begin{remark}\label{rmk:SemisimpleCentreNecessary}
    The assumption that the fusion category $\cat{D}$ is separable is necessary for both the theorem and the corollary above, as the following example shows. We thank Kevin Coulembier for pointing this out to us.
\end{remark}

\begin{example}[Necessity of a semisimple Drinfeld centre] \label{eg:semisimplecentre}
    Let $\cat{D} \coloneqq \vect_{\mathbb{Z}/2\mathbb{Z}}$ and consider the commutative algebra  $A \coloneqq \mathbb{k}[x,t]/\langle t^2, x^2-1-t \rangle \in \vect_{\mathbb{Z}/2\mathbb{Z}}$, with $t$ in degree $0$ and $x$ in degree $1$.
    The Jacobson radical of $A$ is given by $J = \langle t \rangle$, with $A/J \simeq \mathbb{k}[x]/\langle x^2 - 1 \rangle$.
    Observe that the action of $\mathbb{Z}/2\mathbb{Z}$ on $A$ and $A/J$ are both trivial, so $\modda$ has exactly one simple module $A/J$ (up to isomorphism) and one indecomposable projective module which is $A$ itself (as an $A$-module, $A/J$ also has infinite projective dimension).
    
    If $\mathbb{k}$ has characteristic 2, the Drinfeld centre $\cat{Z(D)})$ is not semisimple.
    One way to see this is to use \autoref{lem:dualfusion} and note that $\bimodd(A/J,A/J)^{\otimes op} \simeq \rep(\mathbb{Z}/2\mathbb{Z})$, which is not semisimple due to the characteristic of $\mathbb{k}$.
    Importantly, one can verify that $A \twoheadrightarrow A/J$ has no section that is a graded algebra homomorphism.
    As such, $\mathbb{k}[x]/\langle x^2 - 1 \rangle$ is also a semisimple algebra that is not separable; cf.\ \autoref{prop:separablesection}.
    
    Observe that the $\cat{D}$-module category $\modda$ has only one choice for a $\cat{D}$-basic $\cat{D}$-projective generator, i.e.\ $A$ itself.
    For an equivalence of $\cat{D}$-module categories $\Phi: \modv{B} \to \modda$, where $B$ is $\cat{D}$-basic, we find that $\Phi(B)$ must be a $\cat{D}$-basic $\cat{D}$-projective generator, and hence, again combining \autoref{EnrichmentsKock} and \autoref{enrichedhom}, we obtain isomorphisms $B \simeq \underline{\on{End}}_{B}(B) \simeq \underline{\on{End}}_{A}(\Phi(B))$ of algebra objects in $\cat{D}$. In the case of our chosen algebra $A$, the uniqueness of the $\cat{D}$-projective $\cat{D}$-basic generator for $\modda$ thus entails that 
    $A$ cannot be Morita equivalent to any algebra which is a quotient of the path algebra of a $\cat{D}$-species.
\end{example}

Before we state the next corollary, we need a (most likely well-known) result on categorical rigidity of the regular $\mathbb{Z}_+$modules.
Recall that given a finite $\cat{D}$-module category, its Grothendieck group $\gr{\cat{D}}$ is a $\mathbb{Z}_+$-module over the fusion ring $\gr{\cat{D}}$ of $\cat{D}$ (see \cite[Chapter 3]{EGNO} for the relevant definitions).
We follow the convention that two $\mathbb{Z}_+$-modules over $\gr{\cat{D}}$ are \emph{equivalent} if there is an isomorphism of $\gr{\cat{D}}$-modules between them which is moreover a bijection on the Grothendieck classes of simples (see e.g.\ \cite[Definition 1]{Os}).
\begin{lemma} \label{lem:rigidityofregularmodule}
    Let $\cat{M}$ be an indecomposable semisimple module category over $\cat{D}$.
    Suppose $\gr{\cat{M}}$ is equivalent to $\gr{\cat{D}}$ as $\mathbb{Z}_+$-modules over $\gr{\cat{D}}$.
    Then $M$ is equivalent to $\cat{D}$ as module categories over $\cat{D}$.
\end{lemma}
\begin{proof}
    Let $\on{Irr}(\cat{D}) = \{X_1,X_2,\ldots,X_n\}$ with $X_1 = \mathbb{1}$ and $\on{Irr}(\cat{M}) = \{L_1,L_2,\ldots, L_k\}$.
    Under the assumption $\gr{\cat{D}}$ and $\gr{\cat{M}}$ are equivalent as $\mathbb{Z}_+$-modules over $\gr{\cat{D}}$, we know that $n = k$, and by permuting the indices we may assume that the isomorphism sends $\gr{X_i} \mapsto \gr{L_i}$.
    In particular, $\gr{X_i}\cdot \gr{L_1} = \gr{L_i}$ for all $i$, which implies that $X_i \lact L_1 \simeq L_i$ for all $i$.
    In turn, we get $\underline{\on{End}}_{\cat{M}}(L_1) \simeq \mathbb{1}$ and the result follows from the chain of $\cat{D}$-module equivalences: 
    $
    \cat{M} \simeq \modc{\underline{\on{End}}_{\cat{M}}(L_1)} \simeq \modc{\mathbb{1}} \simeq \cat{D}.
    $
\end{proof}
\begin{remark}
    The above rigidity result on $\mathbb{Z}_+$-modules is specific to the regular $\mathbb{Z}_+$-module. Namely, two module categories $\cat{M}_1$ and $\cat{M}_2$ with equivalent Grothendieck groups (as $\mathbb{Z}_+$-modules) are not necessarily equivalent. See \autoref{eg:Z2divisionalg}.
\end{remark}

\begin{corollary}\label{quiversviaGr}
        Let $\cat{D}$ be a separable fusion category and $\cat{M}$ be a finite $\cat{D}$-module category.
    The following are equivalent:
    \begin{enumerate}
        \item \label{item:Misquiver}
        $\cat{M}$ is equivalent to the category of modules over an admissible quotient of the path algebra of a {\it $\cat{D}$-quiver}.
        \item \label{item:Misregularmodcat}
        $\mathbb{L}(\cat{M})$ is equivalent to $\cat{D}^{\oplus n}$ as $\cat{D}$-module categories.
        \item \label{item:MisregularZ+mod}
        $\gr{\cat{M}} = \gr{\mathbb{L}(\cat{M})}$ is equivalent to $\gr{\cat{D}}^{\oplus n}$ as $\mathbb{Z}_+$-modules over $\gr{\cat{D}}$ for some $n \geq 1$.
        \item \label{item:simplerep}
        There exists $\{L_{1},\ldots,L_{n}\}$ a representative for $\on{\Irr}(\cat{M})/\sim$ such that 
            \begin{itemize}
                \item for all $X\in \on{Irr}(\cat{D})$, $X \lact L_i$ is simple for all $L_i$; and
                \item every simple $L\in \Irr(\cat{M})$ is given by $X \lact L_{i}$ for a unique $X\in \on{Irr}(\cat{D})$ and a unique $L_i$.
            \end{itemize}
    \end{enumerate} 
\end{corollary}
\begin{proof}
    The implication \ref{item:Misquiver} $\implies$ \ref{item:Misregularmodcat} is immediate since $\mathbb{L}(\cat{M})$ will be equivalent to the category of modules over a product of trivial algebras $\mathbb{1}$.
    The equivalence between \ref{item:Misregularmodcat} and \ref{item:MisregularZ+mod} is a direct extension of \autoref{lem:rigidityofregularmodule}.
    The implication \ref{item:MisregularZ+mod} $\implies$ \ref{item:simplerep} is immediate from the definition of equivalent $\mathbb{Z}_+$-modules.
    Finally, \ref{item:simplerep} $\implies$ \ref{item:Misquiver} follows from using the given set of  representatives of $\on{\Irr}(\cat{M})/\sim$ and apply \autoref{extspecies}, where the assumptions ensure that each $\iend{L_i} \simeq \mathbb{1}$.
\end{proof}

\subsection{Hereditary if and only if ideal vanishes}\label{hereditary}

\begin{definition}
 A finite abelian category $\mathcal{A}$ is said to be {\it hereditary} if any subobject of any projective object of $\mathcal{A}$ is itself projective. In other words, the global projective dimension of $\mathcal{A}$ is at most one.
 Accordingly, we say that an algebra in $\cat{D}$ is \emph{hereditary} if its category of modules is such.
\end{definition}

\begin{proposition}\label{hereditarypath}
    Let $(I,(D_{i})_{i \in I}, (E_{ij})_{i,j\in I})$ be an acyclic $\cat{D}$-species. Its path algebra is hereditary.
\end{proposition}

\begin{proof}
    The {\it unfolding} procedure of \cite[Theorem~3.5]{EH} generalises from $\cat{D}$-quivers to $\cat{D}$-species verbatim, showing that the category of modules over its path algebra is equivalent to the representation category of an ordinary (finite) acyclic quiver, which is hereditary.
    We leave the details to the reader.
\end{proof}

\begin{proposition} \label{prop:hereditaryacyclic}
    Let $\cat{M}$ be a finite $\cat{D}$-module category. If $\cat{M}$ is hereditary, then its directed graph, defined in \autoref{defn:internalExtspecies}, is acyclic.
\end{proposition}
\begin{proof}
   In this proof, given $L \in \Irr(\cat{M})$, we will use $[L]$ to denote the equivalence class of $L$ under the equivalence relation $\sim$. 
   Let $\setj{L_{1},\ldots,L_{r}}$ be a set of representatives for $\on{Irr}(\cat{M})/\sim$.
   
   Suppose for contradiction that there is a cycle in the directed graph of $\cat{M}$. 
   In other words, for some $n \leq r+1$ there is a sequence $(i_{1},\ldots,i_{n}) \in \setj{1,\ldots,r}^{n}$ such that $\underline{\on{Ext}}_{\cat{M}}(L_{i_{k}},L_{i_{k+1}}) \neq 0$ for $k=1,\ldots,n$ and $\underline{\on{Ext}}_{\cat{M}}(L_{i_{n}},L_{i_{1}}) \neq 0$. 
   We may further assume that this cycle is minimal, so $L_{i_{j}} \neq L_{i_{j'}}$ for $j\neq j'$. 
   Moreover, by relabelling the indices we can even assume that the cycle is given by $(1,\ldots,n)$.

   We begin by choosing $L'_{1}$ to be a (not necessarily unique) simple such that $L'_1 \sim L_1$ and the projective cover of $L'_1$ is of minimial $\FPdim$ among the projective covers of elements of $[L_{1}]$ -- here, by $\FPdim$ of a module we just mean its $\FPdim$ viewed as an object in $\cat{D}$; see \cite[\S 3.3 and Proposition 4.9.1]{EGNO} for definition and properties of $\FPdim$. 
   In what follows, we shall describe a procedure of choosing simple objects $L'_i$ for each $2 \leq i \leq n$ that satisfy:
   \begin{itemize}
       \item $L'_i \sim L_i$; and
       \item $\on{Ext}_{\cat{M}}(L'_{j},L'_{j+1}) \neq 0$ for all $1 \leq j \leq n-1$.
   \end{itemize}
   Finally, our procedure will also choose another simple $L''_1$ such that $L''_1 \sim L'_1$ ($\sim L_1$) and $\on{Ext}_{\cat{M}}(L'_{n},L''_{1}) \neq 0$.
   
   The procedure is as follows.
   Since $\underline{\on{Ext}}_{\cat{M}}(L_{1},L_{2}) \neq 0$, we have $\underline{\on{Ext}}_{\cat{M}}(L'_{1},L_{2}) \neq 0$ by \autoref{directedgraphindependent}. Thus there is $V \in \cat{D}$ such that $\on{Hom}_{\cat{M}}(V, \underline{\on{Ext}}_{\cat{M}}(L'_{1},L_{2}))\neq 0$, so
   $\on{Ext}_{\cat{M}}(V\lact L'_{1},L_{2}) \simeq \on{Ext}_{\cat{M}}(L'_{1},\rightdual{V}\lact L_{2}) \neq 0$.
   Since $\rightdual{V} \lact L_{2}$ is semisimple, it must contains a simple summand $L'_{2}$ such that $\on{Ext}_{\cat{M}}(L'_{1},L'_{2}) \neq 0$. 
   Note that, by construction, $L'_{2} \sim L_{2}$. 
   We iteratively continue to choose $L'_j$ with the properties $L'_{j} \sim L_{j}$ and $\on{Ext}_{\cat{M}}(L'_{j},L'_{j+1}) \neq 0$ for $j = 1,\ldots,n-1$. 
   Now finally, $\underline{\on{Ext}}_{\cat{M}}(L_{n},L_{1}) \neq 0$ also implies $\underline{\on{Ext}}_{\cat{M}}(L'_{n},L'_{1}) \neq 0$.
   Thus, as before, we can choose some $L''_1$ satisfying $L''_{1} \sim L'_{1}$ ($\sim L_{1}$) and $\on{Ext}_{\cat{M}}(L'_{n},L''_{1}) \neq 0$. 
   
   We are now ready to derive our contradiction using the following general result about hereditary categories: if $\cat{M}$ is hereditary $S$ and $S'$ are two simples in $\cat{M}$ with the property $\Ext_\cat{M}(S,S') \neq 0$, there exists a \emph{strict monomorphism} (i.e.\ it has non-trivial cokernel) of the their projective covers $P' \hookrightarrow P$ (with $P'$ covering $S'$ and $P$ covering $S$); see e.g.\ \cite[Lemma~4.2.2]{Ben}.
   For $i=1,\ldots,n$, let $P'_{i}$ be the projective cover of $L'_{i}$, and let $P''_{1}$ be the projective cover of $L''_{1}$. 
   Our choice of $L'_i$ and $L''_i$ therefore gives us the following chain of strict monomorphisms:
    \begin{equation}\label{strictchain}
        P''_{1} 
        \xhookrightarrow{\not\simeq} P'_{n}
        \xhookrightarrow{\not\simeq} P'_{n-1}
        \xhookrightarrow{\not\simeq} \cdots
        \xhookrightarrow{\not\simeq} P'_{2} 
        \xhookrightarrow{\not\simeq} P'_{1}.
    \end{equation}
    However, this implies that $P''_{1}$ has strictly smaller $\FPdim$ than $P'_{1}$, which contradicts our initial choice of $P'_{1}$ having the minimal $\FPdim$ amongst all other projective covers of simples in $[L_1]$. 
\end{proof}

\begin{lemma}\label{radicaloftensor}
    Let $T_{B}(E)$ be an acyclic species. Then $J(T_{B}(E))$ is the ideal $\left\langle E\right\rangle$ generated by $E$, given by $E \oplus E \otimes_{B} E \oplus \cdots$.
\end{lemma}

\begin{proof}
   This is an immediate consequence of \autoref{radicalofadmissiblequotient}.
\end{proof}

\begin{remark}\label{splittingacyclic}
   In view of \autoref{radicaloftensor}, the radical-splitting of \autoref{radicalsplittingtensor} is particularly simple in the case of an acyclic species: it is given by the morphisms $B \hookrightarrow T_{B}(E)$ and $E \hookrightarrow \left\langle E\right\rangle$.
\end{remark}

\begin{theorem}\label{noidealpaths}
    Let $\cat{D}$ be a separable fusion category.
    Then a finite $\cat{D}$-module category $\cat{M}$ is hereditary if and only if it is equivalent to the category of modules over the path algebra of an acyclic $\cat{D}$-species.
\end{theorem}

\begin{proof}
   One of the directions of the equivalence is given by \autoref{hereditarypath}. It thus suffices to show that if $\cat{M}$ is hereditary, then it is equivalent to the category of modules over the path algebra of an acyclic species in $\cat{D}$.
   By \autoref{extspecies}, we have $\mathcal{M} \simeq \on{mod}_{\cat{D}}\big(T_{B}(E)/I\big)$, where $T_B(E)$ is the path algebra of a $\cat{D}$-species of type $(B,E)$ and $I \subseteq T_{B}(E)_{\geq 2}$.
   Moreover, this $\cat{D}$-species is acyclic by \autoref{directedgraphforspecies} and \autoref{prop:hereditaryacyclic}, so we are left to show that $I$ is trivial.
   
   Using \autoref{radicaloftensor}, it is easy to verify that $J(T_{B}(E))^{2} = T_{B}(E)_{\geq 2}$. Thus, by \autoref{FusionRadical} the ideal $\mathbb{R}(I)$ in $\on{proj}_{\cat{D}}(T_{B}(E))$ that $I$ corresponds to under \autoref{CSZBijection} is contained in $\mathfrak{J}^{2}(\on{proj}_{\cat{D}}(A))$. Similarly to the proof of \autoref{FusionRadical}, we have $\on{proj}_{\cat{D}}(T_{B}(E)/I) \simeq \on{proj}_{\cat{D}}(T_{B}(E))/\mathbb{R}(I)$. Let $P$ be a projective generator for $\on{mod}_{\cat{D}}(T_{B}(E))$ -- not only a projective $\cat{D}$-generator. Then there is an equivalence of abelian categories $\on{mod}_{\cat{D}}(T_{B}(E)) \simeq \on{mod}_{\mathbb{k}}(\on{End}_{T_{B}(E)}(P))$ 
   and
   $\mathbb{R}(I)(P,P)$ is an ideal in $\on{End}_{T_{B}(E)}(P)$, contained in the second power of the Jacobson radical of the latter, such that the quotient is hereditary. This shows that $\mathbb{R}(I)(P,P)$ is zero, by \cite[Lemma~4.2.1]{Ben}, and further shows that $\mathbb{R}(I)$ is zero, and hence so is $I$.
\end{proof}

\begin{corollary} \label{cor:fusionquivercondition}
    Let $\cat{D}$ be a separable fusion category.
    A finite $\cat{D}$-module category $\cat{M}$ is equivalent to the category of modules over the path algebra of an acyclic $\cat{D}$-quiver if and only if $\cat{M}$ is hereditary and any of the equivalent statements in \autoref{quiversviaGr} holds.
\end{corollary}
\begin{proof}
 This is an immediate consequence of \autoref{quiversviaGr} and \autoref{noidealpaths}.
\end{proof}

\begin{example}
    For simplicity, we work over $\mathbb{k}=\mathbb{C}$.
    Consider the following $D_4$ quiver 
    \[
    Q \coloneqq 
    \begin{tikzcd}[row sep=tiny]
        0 \ar[rd] \\
         &  2 \ar[r] & 3 \\
        1 \ar[ru]
    \end{tikzcd}
    \]
    We can define an action of $G \coloneqq \mathbb{Z}/2\mathbb{Z} = \{0,1\}$ on $Q$, where the generator $1 \in G$ acts on vertices by swapping $0$ and $1$, and leaving both $2$ and $3$ fixed.
    This induces an (strong) action of $G$ on $\cat{M} \coloneqq \rep(Q)$; equivalently, $\cat{M}$ is a module category over $\cat{D} \coloneqq \vect_G$.
    By \autoref{noidealpaths}, we know that $\cat{M} \simeq \modda$ for $A$ the path algebra of a $\cat{D}$-species (without further quotient).
    Using \autoref{extspecies}, we can compute the $\cat{D}$-species directly as follows.
    
    We have $\on{Irr}(\cat{M}) = \{S_i \mid 0\leq i \leq 3\}$, one for each vertex $i$.
    Since $S_0$ and $S_1$ are in the same $G$-orbit, and we choose $\on{Irr}(\cat{M})/\sim$ to be represented by $\{S_1, S_2, S_3\}$.
    A direct computation shows that (use e.g.\ \autoref{eqn:internalhomformula} and \autoref{eqn:internalextformula}):
    \begin{equation}
    \iend{S_i} \simeq 
        \begin{cases}
            \mathbb{1}, &\text{if }i=1\\
            D, &\text{if } i=2,3;
        \end{cases} \qquad
    \iExt(S_i,S_j) \simeq 
        \begin{cases}
            \mathbb{k}_0\oplus \mathbb{k}_1, &\text{if } i=1, j=2;\\
            \mathbb{k}_0\oplus \mathbb{k}_1, &\text{if } i=2, j=3; \\
            0, &\text{otherwise},
        \end{cases}
    \end{equation}
    where $D \coloneqq \mathbb{k}[G]$ is the division algebra from \autoref{eg:Z2divisionalg} given by the (untwisted) $G$-graded group algebra.
    By \autoref{extspecies}, the $\cat{D}$-species is: 
    \[
    \begin{tikzcd}
    \mathbb{k}_0 \ar[r, "D"] & D \ar[r,"D"] & D.
    \end{tikzcd}
    \]
    Technically, we have only worked out the underlying objects that label the arrows, not as bimodules. But here it is clear that the bimodules are the obvious ones (note also that objects in $\vect_G$ are self dual).
    Note that this $\cat{D}$-species is exactly \autoref{item:stdZ2action} from \autoref{eg:Z2species}.
    One could define a $\psi$-twisted action of $G$ on the path algebra $\mathbb{k}(Q)$ of the quiver (see \cite{RR_skew}) to obtain a $\vect_G$-module category which is equivalent to the category of modules over the path algebra of the $\cat{D}$-species given in \autoref{item:twistedZ2action} instead.
    Note that while these two module categories are equivalent as abelian categories (both are $\rep(Q)$), they are not equivalent as $\vect_G$-module categories.
    
    We also note that despite the appearance of the $\cat{D}$-quiver in \autoref{item:twoD4copies}, it is substantially different from the other two discussed above. Namely, as abelian category, the category of modules over the path algebra of this $\cat{D}$-quiver is equivalent to a direct sum of two copies of the $D_4$ quiver.

    We also encourage the reader to show that the path algebras of the $\cat{D}$-quiver and the $\cat{D}$-species from \autoref{eg:pentagonCoxquiver} are Morita equivalent by obtaining the $\cat{D}$-quiver from the $\cat{D}$-species.
\end{example}


\begin{thebibliography}{99999999999}
\bibitem[AMS]{AMS} A.~Ardizzoni, C.~Menini, D.~Ştefan, Hochschild cohomology and ``smoothness'' in monoidal categories, J. Pure Appl. Algebra 208 (2007), 297–330.
\bibitem[ASS]{ASS}  I.~Assem, D.~Simson, A.~Skowro{\' n}ski. Elements of the representation theory of associative algebras. Vol. 1. Techniques of representation theory. London Mathematical Society Student Texts, 65. Cambridge University Press, Cambridge, 2006.
\bibitem[Ben]{Ben} D.~J.~Benson, Representations and cohomology, II: Cohomology of groups and modules. Cambridge University Press, 1991
\bibitem[Ber]{Ber} C.~F.~Berg. Structure theorems for basic algebras. arXiv: 1102.1100
\bibitem[Bet]{Bet} A.~Betz. Actions of Fusion Categories on Path Algebras. arXiv:2503.18630
\bibitem[CSZ]{CSZ} K.~Coulembier, M.~Stroi{\' n}ski, T.~Zorman. Simple Algebras and Exact Module Categories. arXiv:2501.06629
\bibitem[DR]{DR} V.~Dlab and C.M.~Ringel. Representations of graphs and algebras, Carleton Mathematical Lecture Notes, No. 8, Carleton University, Department of Mathematics, Ottawa, ON, 1974. iii+86 pp.
\bibitem[DSS]{DSS} C.~L.~Douglas, C.~ Schommer-Pries, Christopher, N.~Snyder, Noah. Dualizable tensor categories,
Mem. Amer. Math. Soc. 268 (2020), no. 1308, vii+88 pp.
\bibitem[EGNO]{EGNO} P.~Etingof, S.~Gelaki, D.~Nikshych, V.~Ostrik. Tensor categories. Mathematical Surveys and Monographs {\bf 205}. American Mathematical Society, Providence, RI, 2015.
\bibitem[EH]{EH} B.~Elias, E.~Heng. Classification of finite type fusion quivers. arXiv:2404.09714
\bibitem[EKW]{EKW} P.~Etingof, R.~Kinser, and C.~Walton. Tensor algebras in finite tensor categories. IMRN 2021 (2021), no. 24, 18529–18572.
\bibitem[EK]{EK} P.~Etingof and M.~Khovanov. Representations of tensor categories and Dynkin diagrams. IMRN 1995 (1995), no. 5, 1073-7928,1687-0247.
\bibitem[EO]{EO} P.~Etingof, V.~Ostrik. Finite tensor categories. Moscow Math. J. {\bf 4}(2004), p.627–654.
\bibitem[Ga1]{Ga1} P.~Gabriel. Indecomposable representations II, Symposia Mathematica, Vol. XI, Academic Press, (1973), 81-104.
\bibitem[Ga2]{Ga2} P.~Gabriel. Unzerlegbare Darstellungen I, Manuscripta Mathematica 6 (1972), pp. 71–103.
\bibitem[Heng]{HengCox} E.~Heng. Coxeter quiver representations in fusion categories and Gabriel's theorem. Selecta Mathematica (N.S.). 30, 67 (2024)
\bibitem[Kel]{Kel} G.~M.~Kelly. Doctrinal adjunction. Category Sem., Proc., Sydney 1972/1973, Lect. Notes Math. 420, 257-280 (1974). 1974.
\bibitem[Ko]{Kock} A.~Kock. Strong functors and monoidal monads. Arch. Math. (Basel) 23 (1972), 113–120.
\bibitem[LMS]{LMS} V.~Linchenko, S.~Montgomery, L.~W.~Small. Stable Jaconson Radical and Semiprime Smash Products, Bulletin of the London Mathematical Society (2005); 37(6):860-872.
\bibitem[KW]{KW} J.~Kesten, C.~Walton. Division algebras in monoidal categories. arXiv:2502.02532v2
\bibitem[KZ]{KZ} L.~Kong, H.~Zheng.  Semisimple and separable algebras in multi-fusion categories. arXiv:1706.06904v2
\bibitem[MM]{MM} V.~Mazorchuk, V.~Miemietz. Transitive 2-representations of finitary 2-categories. Trans. Amer. Math. Soc. 368 (2016), no. 11, 7623--7644.
\bibitem[Os]{Os} V.~Ostrik. Module categories, weak Hopf algebras and modular invariants. Transform. Groups {\bf 8} (2003), no. 2, 177--206.
\bibitem[QZ]{QZ_fusionstable} Y.~Qiu, X.~Zhang. Fusion-stable structures on triangulated categories. Sel.\ Math.\ New Ser.\ 31, 50 (2025). \url{https://doi.org/10.1007/s00029-025-01037-6}
\bibitem[RR]{RR_skew} I.~Reiten, C.~Riedtmann. Skew group algebras in the representation theory of Artin algebras. J. Algebra 92 (1985), no. 1, 224–282.
\bibitem[Sa]{Sa} S.~Sanford. Fusion categories over non-algebraically closed fields. J. Algebra 663 (2025), pp. 316–351.
\bibitem[Sch]{Sch} G.~Schaumann. Fusion Quivers. Published online in Communications in Contemporary Mathematics. \url{https://doi.org/10.1142/S021919972550066X}
\bibitem[St]{Street} R.~Street. Enriched categories and cohomology. Quaestiones Mathematicae, 6:265–283, 1983. Republished as: Reprints in Theory and Applications of Categories, 14 (2005).
\bibitem[SZ]{SZ} M.~Stroi{\' n}ski, T.~Zorman. Reconstruction of module categories in the infinite and non-rigid settings. arXiv:2409.00793
\end{thebibliography}
\end{document}